\documentclass[10pt,a4paper,twoside]{article}

\usepackage[english]{babel}
\usepackage{fullpage}
\usepackage{amsmath,amsfonts,mathtools}%,amsthm,amssymb,mathrsfs}
\usepackage{array}
\newcommand{\ps}[2]{( #1 , #2 )_{\mathcal{X}}}
\newcommand{\norm}[1]{\Vert #1 \Vert_{\mathcal{X}}}
\DeclareMathOperator*{\vspan}{span}
\DeclareMathOperator*{\argmax}{arg\,max}

\DeclareMathOperator*{\dist}{dist}

\DeclareMathOperator*{\ord}{\mathcal{O}}

\newtheorem{theorem}{Theorem}[section]

\newtheorem{corollary}[theorem]{Corollary}

\newtheorem{lemma}[theorem]{Lemma}

\newenvironment{proof}[1][Proof]{\noindent\textbf{#1.} }{\ \rule{0.5em}{0.5em}}
\begin{document}
%%%%%%%
% Titel
\title{Convergence analysis of the Generalized Empirical Interpolation Method}
\author{Y. Maday$^{1, 2, 3}$, O. Mula$^{4}$ and G. Turinici$^{2,4}$\\
{$^1$ {\small UPMC Univ Paris 06, UMR 7598, Laboratoire Jacques-Louis Lions, F-75005, Paris, France.}}\\
{$^2$ {\small Institut Universitaire de France.}}\\
{$^3$ {\small Brown Univ, Division of Applied Maths, Providence, RI, USA.}}\\
{$^4$ {\small Universit\'e Paris-Dauphine, PSL Research University, CNRS, UMR 7534, CEREMADE, 75016 Paris, France}}\\
}
\date{}
\maketitle

\paragraph*{Abstract:} Let $F$ be a compact set of a Banach space $\mathcal{X}$. 
This paper analyses the ``Generalized Empirical Interpolation Method'' (GEIM) which, given a function $f\in F$, 
builds an interpolant $\mathcal{J}_n[f]$ in an $n$-dimensional subspace $X_n \subset \mathcal{X}$ with the knowledge of $n$ outputs $(\sigma_i(f))_{i=1}^n$, where $\sigma_i\in \mathcal{X}'$ and $\mathcal{X}'$ is the dual space of $\mathcal{X}$. 
The space $X_n$ is built with a greedy algorithm that is \textit{adapted} to $F$ in the sense that it is generated by elements of $F$ itself. The algorithm also selects the linear functionals $(\sigma_i)_{i=1}^n$ from a dictionary $\Sigma\subset \mathcal{X}'$.
In this paper, we study the interpolation error $\max_{f\in F} \Vert f-\mathcal{J}_n[f]\Vert_{\mathcal{X}}$ by comparing it with the best possible performance on an $n$-dimensional space, i.e., the Kolmogorov $n$-width of $F$ in $\mathcal{X}$, $d_n(F,\mathcal{X})$. For polynomial or exponential decay rates of $d_n(F,\mathcal{X})$, we prove that the interpolation error has the same behavior modulo the norm of the interpolation operator. Sharper results are obtained in the case where $\mathcal X$ is a Hilbert space.

\section{Introduction}
\label{section::intro}
Let $\mathcal{X}$ be a Banach space of functions defined over a domain $\overline{\Omega} \subset \mathbb{R}^d$ or $\mathbb{C}^d$, $d \geq 1$, and let $F$ be a compact set of $\mathcal{X}$. Without loss of generality, we assume that the functions $\varphi \in F$ satisfy $\norm{\varphi}\leq 1$, where $\norm{\cdot}$ is the norm of $\mathcal{X}$. In this paper, we investigate the approximation quality of functions in $F$ with the ``Generalized Empirical Interpolation Method'' (GEIM, see \cite{MagenesGEIM,mulaStokesGEIM}). For a given $f\in F$, the method builds an interpolant $\mathcal{J}_n[f]$ in an $n$-dimensional subspace $X_n \subset \mathcal{X}$ with the knowledge of $n$ outputs of $(\sigma_i(f))_{i=1}^n$ where the $\sigma_i$ are bounded linear functionals of $\mathcal{X}'$. The approximation space $X_n$ is built with a greedy algorithm that is \textit{adapted} to $F$ in the sense that it is generated by some elements of $F$ itself. The algorithm also selects the linear functionals to be used from a dictionary $\Sigma\subset \mathcal{X}'$. This procedure is a generalization of the Empirical Interpolation Method (see \cite{Barrault,GreplMagic,MadayMagic}) that was originally defined for $(\mathcal{C}(\bar \Omega),\Vert\cdot\Vert_{\infty})$ and with Dirac masses as linear functionals.

The current setting falls into in the framework of optimal recovery (see, e.g.~\cite{MRW1976,MR1985}) in the sense that, for a given $f\in F$, we want to approximate it in an appropriate basis and using the knowledge of certain outputs $\sigma_i(f)$, $1\leq i \leq n$. However note that there is an important difference with respect to the classical setting of this field which assumes that $F$ is the unit ball of a smoothness space in $\mathcal{X}$. In the current setting, we allow $F$ to have more general types of geometry/regularity. For this reason, our approximation space $X_n$ and the linear functionals are chosen depending on $F$. 
This is in contrast to other methods like polynomial/spline/radial interpolation (see, e.g., \cite{Buhmann2000}) or even meshless methods \cite{Belytschko1996} since there the basis functions are fixed in advance. This idea of adaptivity makes GEIM be also different from kriging \cite{Stein2012} where an underlying stochastic process is a priori given. 
Here we only assume that $F$ is compact, which enlarges the range of potential applications. For instance, it is possible to apply the methodology to the case where $F$ is the solution manifold of a parametric PDE. This case is relevant for at least two types of applications. The first concerns the approximation of this manifold with {\it reduced basis} 
methods and the treatment of nonlinearities in the PDE (see \cite{GreplMagic}). The second concerns the coupling of measurement data (represented by the $\sigma_i(f)$, $1\leq i \leq n$) with parametrized models (represented by the space $X_n$) in a systematic way and on the basis of suitable functional spaces. For the particular setting of GEIM, 
numerical examples can be found in \cite{MagenesGEIM,mulaStokesGEIM} where it was also explained how the method could assist in the placement of sensors in real physical experiments and the minimization of their number. After the submission of this paper, this line of research has been further developed in several relevant works. For Hilbert spaces, an extension has been proposed in \cite{pateraYanoPaper} which consists in a particular least squares approximation with $m\geq n$ measurement data. In addition, it has been proven in \cite{BCDDPW2015} that, when $\mathcal{X}$ is a Hilbert space, both GEIM and the least-squares method are optimal in a sense that will be clarified later in this paper. Last but not least, we would like to cite the even more recent work \cite{DPW2016} on data assimilation in Banach spaces which can be seen as a further abstraction of GEIM and \cite{pateraYanoPaper,BCDDPW2015}.

Since the interpolation operator $\mathcal{J}_n:\mathcal{X}\to X_n$ built by the GEIM targets the elements of $F$, it is important to quantify the error
\begin{equation}
\label{eq:worstInterpError}
\max_{u \in F} \Vert u - \mathcal{J}_n[u] \Vert_{\mathcal{X}}.
\end{equation}
The aim of this paper is to compare \eqref{eq:worstInterpError} with the best possible performance in an $n$ dimensional subspace of $\mathcal{X}$, which is given by the Kolmogorov $n$-width of $F$ in $\mathcal{X}$,
\begin{equation}
\label{eq::kolmoWidth}
d_n(F,{\cal X}) \coloneqq \underset{\underset{\dim(X_n)\leq n}{X_n \subset \cal{X}}}{\inf}\ \underset{u \in F}{\max}\ \underset{v \in X_n}{\inf} \ \Vert u-v \Vert_{\cal{X}}.
\end{equation}
In particular, we show that when $d_n(F,\mathcal{X})$ decays polynomially or exponentially, the interpolation error has the same behavior modulo a coefficient depending on the norm of the interpolation operator
\begin{equation}
\label{conv::eq::lebesgueCt}
\Lambda_n \coloneqq \underset{\varphi\in \mathcal{X}}{\sup} \dfrac{\Vert {\cal J}_n[\varphi] \Vert_{\mathcal{X}}}{\Vert \varphi \Vert_{\mathcal{X}}},
\end{equation}
called Lebesgue constant. The conditions required on $F$ to guarantee a certain decay rate in the sequence $\left(d_n(F,\mathcal{X})\right)_{n\ge 1}$ are far from trivial and require some regularity of the functions of $F$ (see \cite{Pinkus1985} for a general reference). For the particular case when $F$ is the solution manifold of a parametric PDE, we refer to \cite{cohenDeVoreKolmoHolomo} in which the decay of $\left(d_n(F,\mathcal{X})\right)_{n\ge 1}$ is connected with the regularity of the manifold with respect to the parameters.

The paper is organized as follows: in section \ref{section::greedy} we explain the approximation strategy and the greedy algorithm of the GEIM. We also show that, under appropriate hypothesis, the greedy algorithm is of a weak type in the sense of \cite{deVoreCohenApriori}. In section \ref{section::practice}, we make several remarks on its practical implementation. Then we present the main results of the paper and  derive convergence rates of the interpolation error in Banach spaces in section \ref{section::banach}. In section \ref{section::hilbert}, improved results are given for the particular case of Hilbert spaces. Note that results on convergence rates of GEIM have already been given in \cite{mulaSampta2013} for $\mathcal{X}=L^2(\Omega)$ and the main contribution of this work is to extend the results to general Banach spaces. Our methodology is based on the works \cite{deVoreCohenApriori,devore2012greedy} on convergence rates of reduced basis and should be seen as a generalization of them in a sense that will be specified in \ref{sec:strategyRates}.

\section{The Generalized Empirical Interpolation Method}
\label{section::greedy}
In what follows, we assume that the dimension of the vector space spanned by $F$ is larger than $\mathcal{N}$, where $\mathcal{N}$ is some given large number depending on $F$ ($\mathcal{N}$ could potentially be infinite). This hypothesis is made since the asymptotic decay of the interpolation error is trivial otherwise. 
We also suppose given a dictionary $\Sigma \subset \mathcal{X}'$ of bounded linear functionals with the following properties:
\begin{itemize}
\item[P1:] $\forall\sigma \in \Sigma$, $\Vert \sigma \Vert_{\mathcal{X}'}=1$.
\item[P2:] \textit{Unisolvence}: If $\varphi \in \vspan\{F\}$ is such that $\sigma(\varphi)=0,\ \forall \sigma \in \Sigma$, then $\varphi = 0$.
\item[P3:] $\forall \varphi \in \vspan\{F\}$ the set  $\{ \sigma(\varphi), \sigma \in \Sigma \} \subset \mathbb{R}$ is closed.
\end{itemize}
For all $n < \mathcal{N}$, our goal is to build a sequence of $n$-dimensional subspaces of $\mathcal{X}$ : $X_n = \vspan \{ \varphi_i \in \mathcal{X}\}_{i=0}^{n-1}$, that approximate well enough the elements of $F$. As already mentioned in the introduction, one of the key features that makes the approximation adapted to $F$ is that the basis functions $\varphi_i$, $0\leq i \leq n-1$, are chosen in $F$ itself. The approximation in $X_n$ of any $\varphi \in \mathcal{X}$ is done with its so-called generalized interpolant $\mathcal{J}_n[\varphi] \in X_n$ which satisfies the following interpolation property for a set of ``well-chosen'' bounded linear functionals $\{ \sigma_0$, $\sigma_{1}$,\dots , $\sigma_{n-1} \}$ of $\Sigma$:
\begin{equation}
\label{conv::eq::eqGEIM}
\mathcal{J}_n[\varphi] = \sum_{j=0}^{n-1} \beta_j \varphi_j, \quad\text{such that } \quad \sigma_i({\cal J}_n[\varphi]) = \sigma_i(\varphi),\ \forall i=0,\dots, n-1.
\end{equation}
The construction of the spaces $X_n$ and the selection of the suitable bounded linear functionals of $\Sigma$ is recursively carried out by a greedy algorithm where the basis functions are selected one after another in $F$ and tuned to improve the accuracy of the interpolation process. In general, from a practical point of view, since $F$ is an infinite compact set, the $n^{th}$ function is chosen in a finite, large enough set $F_n$, where $\{F_n\}_{n\in \mathbb{N}}$ is an embedded sequence of subsets of $F$. Similarly $\{\widetilde \Sigma_n\}_{n\in \mathbb{N}}$ is an embedded sequence of finite subsets of $\Sigma$. For any $n\in \mathbb{N}$, $\widetilde \Sigma_n$ satisfies P2 and P3 for the set $F_n$. In section \ref{section::practice}, we discuss how $F_n$ can be chosen in practice. At the initialization of the greedy algorithm, the first interpolating function $\varphi_0$ is chosen as
\begin{equation}
\label{eq:varphi0}
\varphi_0  =  \argmax_{\varphi \in F_0} \norm{\varphi}, 
\end{equation}
and we define $X_1 \coloneqq \vspan \{  \varphi_0\}$. The first interpolating bounded linear functional is 
\begin{equation*}
\sigma_0 = \argmax_{\sigma\in  \widetilde \Sigma_0} \vert \sigma (\varphi_0) \vert .
\end{equation*}
The interpolation operator $\mathcal{J}_1 : \mathcal{X} \mapsto X_1$ is defined such that \eqref{conv::eq::eqGEIM} is true for $n=1$, i.e., $ \sigma_0 (\mathcal{J}_1 [\varphi]) = \sigma_0(\varphi)$, for any $\varphi \in \mathcal{X}$. To facilitate the computation of the generalized interpolant, we express it in terms of $q_0 \coloneqq \varphi_0 / \Vert \varphi_0 \Vert_{\mathcal{X}}$. In this basis, the interpolant reads
$\mathcal{J}_1[\varphi] = \sigma_0(\varphi) q_0$, for any $\varphi \in \mathcal{X}$. We then proceed by induction. Assume that, for a given $1\leq n < \mathcal{N}$, we built  the set of interpolating functions $\{  q_0,  q_1,\dots,  q_{n-1} \}$ and the  set of associated interpolating bounded linear functionals $\{\sigma_0,\sigma_1,\dots,\sigma_{n-1} \}$ such that the operator $ \mathcal{J}_{n}[\varphi]=\sum\limits_{j=0}^{n-1} \alpha_j^{n}(\varphi)  q_j$, 
is well defined for any $\varphi \in \mathcal{X}$ and the coefficients $\{\alpha_j^{n}(\varphi)\}_{j=0,\dots,n-1}$, are given by the interpolation problem
\begin{equation*}
\begin{cases}
\text{Find } \left( \alpha_j^{n}(\varphi) \right)_{j=0}^{n-1} \text{ such that:}  \\
\sum\limits_{j=0}^{n-1} \alpha_j^{n}(\varphi) \sigma_i( q_j) = \sigma_i (\varphi),\quad \forall i=0,\dots,n-1 .
\end{cases}
\end{equation*}
We now define $\varepsilon_n (\varphi) \coloneqq \Vert \varphi - {\cal J}_{n}[\varphi] \Vert_{\mathcal{X}}$, for any $\varphi \in \mathcal{X}$, 
and  choose $\varphi_{n}$ such that
\begin{equation}
\label{2K10}
\varphi_{n}
= \argmax_{\varphi \in F_n }  \varepsilon_{n}(\varphi) 
\end{equation}
and then $\sigma_n$ such that
\begin{equation*}
\sigma_{n}= \argmax_{\sigma \in \widetilde \Sigma_{n}} \vert \sigma(\varphi_{n}-{\cal J}_{n}[\varphi_{n}])  \vert .
\end{equation*}
Note that the existence of $\sigma_{n}$ is ensured by the property P3 of the dictionary $\widetilde \Sigma_n$.
The next basis function is then $q_{n}= \left( \varphi_{n}-{\cal J}_{n}[\varphi_{n}] \right) /  \sigma_{n}(\varphi_{n}-{\cal J}_{n}[\varphi_{n}])$. We finally set $ X_{n+1} \coloneqq \vspan\{  q_j,\ j\in [0,n] \} = \vspan\{ \varphi_j,\ j\in [0,n] \}$. The interpolation operator $\mathcal{J}_{n+1}:\mathcal{X} \mapsto X_{n+1}$ is given by
\begin{equation*}
\forall \varphi \in \mathcal{X},\quad \mathcal{J}_{n+1}[\varphi]=\sum\limits_{j=0}^{n} \alpha_j^{n+1}(\varphi)  q_j
\end{equation*}
and the coefficients $\alpha_j^{n+1}(\varphi)$, $j=0,\dots,n$, are given by the interpolation problem
\begin{equation*}
\begin{cases}
\text{Find } \left( \alpha_j^{n+1}(\varphi) \right)_{j=0}^{n} \text{ such that:}  \\
\sum\limits_{j=0}^{n} \alpha_j^{n+1}(\varphi) \sigma_i( q_j) = \sigma_i (\varphi),\quad \forall i=0,\dots,n.
\end{cases}
\end{equation*}
It has been proven in \cite{MadayMagic} (for EIM) and \cite{mulaStokesGEIM} (for GEIM) that for any $1\leq n\leq \mathcal{N}$, the set $\{ q_j,\ j\in [0,n-1] \}$ is linearly independent and that this interpolation procedure is well-posed in $\mathcal{X}$. This follows from the fact that the matrix $\left( \sigma_i (q_j)\right)_{0\leq i,j \leq n-1} $ is lower triangular with diagonal entries equal to $1$. Using the triangular inequality it is standard to derive the following inequality on the interpolation error
\begin{equation}
\label{conv::eq::interpError}
\forall \varphi \in \mathcal{X},\quad
\Vert \varphi-{\cal J}_n[\varphi] \Vert _{\mathcal{X}} \leq (1+\Lambda_n) \underset{\psi_n \in X_n}{\inf}\Vert \varphi-\psi_n \Vert _{\mathcal{X}} ,
\end{equation} 
where $\Lambda_n$ is the Lebesgue constant, i.e., the norm of the interpolation operator $\mathcal{J}_n:\mathcal{X}\to X_n$ defined in \eqref{conv::eq::lebesgueCt}. Note that nothing at this point ensures that $\sup_{\varphi \in F} \Vert \varphi-{\cal J}_n[\varphi] \Vert _{\mathcal{X}}$ is close to $d_n(F,\mathcal{X})$. One of the main reasons is that the basis set of $X_n$ has been derived in a hierarchical manner with the greedy algorithm. It is actually the purpose of this paper to connect these two quantities. As we will see in sections \ref{section::banach} and \ref{section::hilbert}, the Lebesgue constant will be involved in the bounds and it is therefore important to discuss its behavior as $n$ increases. First of all, $\Lambda_n$ depends both on $F$ and $\Sigma$. In a practical implementation, it also depends on the choice of the subsets $F_i$ and $\widetilde \Sigma_i,\ 0\leq i\leq n-1$.  This observation is particularly clear in the case of Hilbert spaces where $\Lambda_n = 1/ \beta_n$, with
\begin{equation*}
\beta_n \coloneqq \inf_{\varphi \in X_n}  \sup_{\sigma \in \vspan \{ \sigma_0,\dots ,\sigma_{n-1}\}}  \dfrac{\langle  \varphi,\sigma \rangle_{\mathcal{X} , \mathcal{X}'}}{ \Vert \varphi \Vert_{\mathcal{X}}  \Vert \sigma  \Vert_{\mathcal{X}'} }
\end{equation*}
and the bound \eqref{conv::eq::interpError} holds without the 1 on the right hand side (see \cite{mulaStokesGEIM}). In addition, it has recently been proven in \cite{BCDDPW2015} that, in this setting, $1/\beta_n$ is the smallest constant to relate $\max_{\varphi\in F} \Vert \varphi-{\cal J}_n[\varphi] \Vert _{\mathcal{X}}$ and $\max_{\varphi\in F} \inf_{\psi_n \in X_n} \Vert \varphi-\psi_n \Vert _{\mathcal{X}}$. A (generally very pessimistic) bound for $\Lambda_n$ that does not take into account the dependence on $F$ and $\Sigma$ was derived in \cite{MagenesGEIM}  and reads
\begin{equation}
\label{eq:boundLebesgue}
\Lambda_n \leq 2^{n-1} \max_{i \in [0,n-1]} \Vert q_i \Vert_{\mathcal{X}}.
\end{equation}
In \cite{MadayMagic}, an example that achieves the bound \eqref{eq:boundLebesgue} was built but it involves a set $F$ with large $n$-width in $\mathcal{X}$. This does not correspond to the current setting since we assume that the sequence $(d_n(F,\mathcal{X}))$ has a relatively fast decay rate. In this context, numerical evaluations  of $\Lambda_n$ indicate that it increases polynomialy (with small degree) in the worst case scenario (see, e.g., \cite{Barrault}, \cite{GreplMagic}, \cite{MadayMagic} for examples in the case of the EIM and \cite{mulaStokesGEIM} for GEIM). For this reason, the conjecture that $\Lambda_n$ depends mildly on the dimension when the $n$-width of $F$ is small and/or decays with $n$ seems reasonable. 

\section{Some comments on the practical implementation of the greedy algorithm}
\label{section::practice}
Let us now discuss the choice of the finite sets $F_n$ for $n\in \mathbb{N}$. These are to be chosen large enough such that the search over $F_n$  of $\argmax_{\varphi \in F_n }  \varepsilon_{n}(\varphi) $ (in \eqref{2K10}) is ``close enough'' to the same search over $F$. The following lemma quantifies this idea:
\begin{lemma}
\label{conv::lemma::existenceMesh}
For any $0<\eta < 1$, there exits a family 
$\{ F_n \equiv F^\eta_n\}_{n=0}^{\mathcal{N}  }$ of finite subsets 
of $F$ such that the greedy algorithm satisfies
\begin{equation}
\label{conv::eq::existenceMeshEqs}
\begin{cases}
\underset{\varphi \in F^\eta_0}{\max} \norm{\varphi}
\geq
\eta\ \underset{\varphi \in F}{\max} \norm{\varphi}, \\
\underset{\varphi \in F^\eta_n }{\max} \norm{\varphi -  \mathcal{J}_n[\varphi]}
\geq \eta\ \underset{\varphi \in F}{\max} \norm{\varphi -  \mathcal{J}_n[\varphi]},\quad \forall\ n \in \{1,\dots,\mathcal{N}\} .
\end{cases}
\end{equation}
\end{lemma}
 \begin{proof}
We first address the case $n=0$ as follows. Let $\eta_0 = \eta$. By compactness of $F$, there exists a finite subset $F^{\eta}_0 \subset F$ and a function $\tilde{\varphi}_0 \in F$ such that
\begin{equation}
\label{eq:varphi10}
\underset{\varphi \in F}{\max}\ \underset{\psi \in F^{\eta}_0}{\min}  \norm{\varphi - \psi} \leq (1-\eta_0) \norm{\tilde{\varphi}_0}.
\end{equation}
Let $\varphi_0 = \underset{\psi \in F^{\eta}_0}{\argmax} \norm{\psi}$ and 
$\varphi^{\max}_0 = \underset{\varphi \in F}{\argmax} \norm{\varphi}$. For any $\psi\in F^{\eta}_0$,
\begin{equation*}
\norm{\varphi_0 } 
\geq 
\norm{\psi} \geq -\norm{\psi - \varphi^{\max}_0} + \norm{ \varphi^{\max}_0 },
\end{equation*}
from which we infer that
\begin{equation*}
\norm{\varphi_0 } 
\geq 
-\min_{\psi\in F^{\eta}_0} \norm{\psi - \varphi^{\max}_0} + \norm{ \varphi^{\max}_0 },
\end{equation*}
which, from \eqref{eq:varphi10}, yields
\begin{equation*}
\norm{\varphi_0 } 
\geq
-(1-\eta_0) \norm{\tilde{\varphi}_0} + \norm{\varphi^{\max}_0}
\geq \eta_0 \norm{\varphi^{\max}_0}.
\end{equation*}
This completes the proof of the first inequality of \eqref{conv::eq::existenceMeshEqs}.
For any $1\leq n \leq \mathcal{N}$, let 
\begin{equation}
\label{eq::etaN}
\eta_n = 1- (1-\eta)/(1+\Lambda_n).
\end{equation}
We define $F^{\eta}_n \coloneqq F^{\eta }_{n-1} \cup \Xi^{\eta_{n} }_{n}$, where $\Xi^{\eta_{n} }_{n}$ is a finite subset of $F$ such that 
\begin{equation}
\label{eq:gridN}
\underset{\varphi \in F}{\max}\ \underset{\psi \in \Xi^{\eta_n}_n}{\min}  \norm{\varphi - \psi} \leq (1-\eta_n) \norm{r_n[\tilde{\varphi}_n]},
\end{equation}
for some $\tilde \varphi_n \in F$ and where $r_n[\varphi] \coloneqq \varphi - \mathcal{J}_n[\varphi],\ \forall \varphi \in \mathcal{X}$. The existence of $\Xi^{\eta_{n} }_{n}$ and $\tilde \varphi_n$ follows from the compactness of $F$, the fact that $r_n: \mathcal{X} \to \mathcal{X}$ is continuous (with a norm that is upper-bounded by $1+\Lambda_n$) and that $r_n(F)$ is a compact subset of $\mathcal{X}$. Using that $ \Xi^{\eta_{n} }_{n} \subset F^{\eta}_n $ and inequality \eqref{conv::eq::interpError}, we derive
\begin{align*}
\underset{\varphi \in F}{\max}\ \underset{\psi \in F^{\eta}_n}{\min}  \norm{r_n[\varphi-\psi]}  
\leq \underset{\varphi \in F}{\max}\ \underset{\psi \in \Xi^{\eta_n}_n}{\min}  \norm{r_n[\varphi-\psi]}  
\leq (1+\Lambda_n) \underset{\varphi \in F}{\max}\ \underset{\psi \in \Xi^{\eta_n}_n}{\min}  \norm{\varphi - \psi}.
\end{align*}
Thus, by using \eqref{eq:gridN} in the previous inequality, we have
\begin{align}
\underset{\varphi \in F}{\max}\ \underset{\psi \in F^{\eta}_n}{\min}  \norm{r_n[\varphi-\psi]}  
\leq (1+\Lambda_n)(1-\eta_n) \norm{r_n[\tilde{\varphi}_n]}   
= (1-\eta) \norm{r_n[\tilde{\varphi}_n]}.   \label{eq:gridAux}
\end{align}
Let $\varphi_n = \underset{\psi \in F^{\eta}_n}{\argmax} \norm{r_n[\psi]}$ and 
$\varphi^{\max}_n = \underset{\varphi \in F}{\argmax} \norm{r_n[\varphi]}$. For any $\psi\in F^{\eta}_n$,
\begin{align*}
\norm{r_n[\varphi_n]} 
& \geq 
\norm{r_n[\psi]} 
\geq 
-\norm{  r_n[\psi- \varphi^{\max}_n ]} + \norm{ r_n[\varphi^{\max}_n] } ,
\end{align*}
which, by using inequality \eqref{eq:gridAux}, finally yields
\begin{align*}
\norm{\varphi_n - \mathcal{J}_n[\varphi_n] } 
& \geq  
-(1-\eta) \norm{r_n[\tilde{\varphi}_n]} + \norm{r_n[\varphi^{\max}_n]} 
\geq 
\eta\ \norm{\varphi^{\max}_n - \mathcal{J}_n[\varphi^{\max}_n]} .
\end{align*}
which ends the proof.
\end{proof}

The parameter $\eta$ above quantifies to what extent the search over $F_n = F^{\eta}_n$ differs from the search over $F$ in the greedy algorithm. This relaxation expressed in the form of \eqref{conv::eq::existenceMeshEqs} is known as the weak greedy algorithm (in the sense defined in section 1.3 of \cite{deVoreCohenApriori}). We set $\eta=1$ as an extreme case where $F^{\eta}_n = F$ (this is the strong greedy algorithm). Then, the smaller the $\eta$, the coarser the search over $F^{\eta}_n$ will be in comparison with a search over $F$. 

An important point to note is that the construction above depends on the application of the finite covering property of compact sets. Hence the question: how to obtain the sets $\{F_n\}_{n=0}^{\mathcal{N}}$ in practice? In full generality, this task is not entirely possible since it requires optimizations over the whole set $F$. However, the problem becomes feasible if we have some additional knowledge of the manifold $F$ like, e.g., information about its geometry or regularity. As an example, let us consider the case where $F$ is a set of parameter dependent functions $F = \{ u(.;\mu) ; \mu\in {\mathcal D}\}$ where ${\mathcal D}$ is a compact set of parameters\footnote{Note that in the application of GEIM discussed in section \ref{section::intro}, the compact set $F$ is of this form.}. Then the derivability of the mapping $\mu\in {\mathcal D}\mapsto  u(.;\mu)$ and a known uniform bound on this derivative with respect to $\mu$ (no regularity in the spacial direction is assumed here) allows to build a finite covering from a finite set in the compact set ${\mathcal D}$ in a completely constructive way.

Instead of working with such certified a priori adaptive subsets, another a posteriori adaptative option can be proposed following the arguments presented in \cite{madayStammLocallyAdaptive} where a knowledge of the geometry of $F$ is learnt on the fly as the greedy algorithm is implemented.

Note finally that, in many  actual implementations, a less ideal approach is used where a fixed, unique, large enough, subset $F_{finite}\subset F$ (and a fixed subset $\widetilde \Sigma$) is chosen. In frequent cases where the Kolmogorov dimension is rapidly decaying to zero the greedy algorithm ends after very few iterations and this crude procedure actually works well in practice.

%%%%%%%%%%%%%%%%%%%%%%%%%%%%%%%%%%%%%%%%%%%%%%%%%%%%%%%%%%%
\section{Convergence rates of the GEIM in a Banach space}
\label{section::banach}

In order to have consistent notations in what follows, we define $\varphi_n =0$ and $X_n = X_{\mathcal{N}}$ for $n>\mathcal{N}$.
In this section, $\mathcal{X}$ is a Banach space.

\subsection{Preliminary notations and properties}

To fix some notations, let $K$ be a nonempty subset of $\mathcal{X}$. For every $\varphi \in \mathcal{X}$, the distance between $\varphi$ and the set $K$ is 
\[
\dist(\varphi,K) \coloneqq \underset{ y \in K}{ \inf } \norm{\varphi-y}.
\]
For any $\varphi \in \mathcal{X}$, the metric projection of $\varphi$ onto $K$ is the set 
\begin{eqnarray*}
P_K(\varphi) = \{ z \in K : \norm{\varphi-z} = \dist(\varphi,K) \}.
\end{eqnarray*}
In general, this set can be empty or composed of one or more than one element. However, in the particular case where $K$ is a finite dimensional vector space, $P_K(\varphi)$ is not empty. For any $n\geq 1$, the non empty set
\begin{equation}
\label{conv::eq::metricProj}
P_n(\varphi) = \{ z \in X_n : \norm{\varphi-z} = \dist(\varphi,X_n) \}
\end{equation}
will denote the metric projection of $\varphi \in \mathcal{X}$ onto $X_n$. Since the uniqueness of the metric projection onto $X_n$ is  not necessarily ensured, in the following, $P_n(\varphi)$ will denote one of the elements of the set \eqref{conv::eq::metricProj}. We now define
\begin{equation}
\label{eq::tau}
\tau_n(F)_{\mathcal{X}}:= \max_{f \in F }\Vert f-P_n(f) \Vert_{\mathcal{X}}, \quad n\geq 1.
\end{equation}
Note that $\tau_n(F)_{\mathcal{X}} \leq \max_{f \in F }\Vert f \Vert_{\mathcal{X}} \leq 1$ given that the elements of $F$ have norm less than $1$. 
We will use the abbreviation $\tau_n$ and $d_n$ for  $\tau_n(F)_{\mathcal{X}}$ and $d_n(F,\mathcal{X})$. Likewise, $(\tau_n)_{n=1}^\infty$ and $(d_n)_{n=1}^\infty$ will denote the sequences $\left( \tau_n(F)_{\mathcal{X}} \right)_{n=1}^{\infty}$ and $\left( d_n(F,\mathcal{X}) \right)_{n=1}^{\infty}$ respectively. Finally we introduce the parameter
\begin{equation}\label{3K11}
\gamma_n= \dfrac{\eta}{1+\Lambda_n}, \quad \forall\ 1\leq n  \leq \mathcal{N}.
\end{equation}
where $\eta$ was introduced in \eqref{conv::eq::existenceMeshEqs} and $\Lambda_n$ is the Lebesgue constant.

\subsection{Main strategy to derive convergence rates}
\label{sec:strategyRates}
We start this section by proving the following lemma.
\begin{lemma}
\label{lemma::weakGreedy}
For any $n \geq 1$, the function $\varphi_{n}$ defined in \eqref{2K10} verifies
\begin{equation}
\label{eq:lemma::weakGreedy}
\Vert \varphi_{n} - P_n(\varphi_{n}) \Vert_{\mathcal{X}} \geq  \gamma_n \tau_n.
\end{equation}
\end{lemma}
\begin{proof}
\label{proof::weakGreedy}
From equation \eqref{conv::eq::interpError} applied to $\varphi = \varphi_{n}$ we have $\Vert \varphi_{n} - P_n(\varphi_{n}) \Vert_{\mathcal{X}} \geq \dfrac{1}{1+\Lambda_n} \Vert \varphi_{n} - {\cal J}_n(\varphi_{n}) \Vert_{\mathcal{X}}$. But $\Vert \varphi_{n} - {\cal J}_n(\varphi_{n}) \Vert_{\mathcal{X}} \geq \eta \Vert \varphi - {\cal J}_n(\varphi) \Vert_{\mathcal{X}}$ for any $\varphi \in F$ according to the definition of $\varphi_{n}$. Thus $\Vert \varphi_{n} - P_n(\varphi_{n}) \Vert_{\mathcal{X}} \geq  \gamma_n \Vert \varphi - {\cal J}_n(\varphi) \Vert_{\mathcal{X}} \geq  \gamma_n \Vert \varphi-P_n(\varphi) \Vert_{\mathcal{X}}$.
\end{proof}

Lemma \ref{lemma::weakGreedy} shows that the weak greedy algorithm of GEIM has very similar properties as the one in
\cite{devore2012greedy}. The difference is that, in \cite{devore2012greedy}, inequality \eqref{eq:lemma::weakGreedy} involves a constant parameter $\gamma$ independent of $n$. 
We take this into account in our analysis: in section \ref{subsection::banach::ratesProjection} we analyze the convergence rates for $(\tau_n)_{n=1}^\infty$ by extending the proofs of \cite{devore2012greedy} to the case where $\gamma$ depends on $n$. For the sake of comparison, we first recall here their main results in lemmas \ref{conv::lemma:deVoreBanachPoly} and \ref{conv::lemma:deVoreBanachExp} below:

\begin{lemma}[Corollary $4.2-(ii)$ of \cite{devore2012greedy}]
\label{conv::lemma:deVoreBanachPoly}
If, for $\alpha>0$, we have $d_n \leq C_0 n^{-\alpha}$, $n=1,2,\dots$, then for any $0 < \beta < \min\{ \alpha, 1/2 \}$, we have $\tau_n \leq C_1 n^{-\alpha+1/2+\beta}$, $n=1,2,\dots$, with
\begin{equation*}
C_1 := \max\left\lbrace  C_0 4^{4\alpha +1} \gamma^{-4} \left( \dfrac{2\beta+1}{2\beta} \right)^{\alpha}  ; \underset{n=1,\dots,7}{\max} n^{\alpha-\beta-1/2} \right\rbrace .
\end{equation*}
\end{lemma}

\begin{lemma}[Corollary $4.2-(iii)$ of \cite{devore2012greedy}]
\label{conv::lemma:deVoreBanachExp}
If, for $\alpha>0$, $d_n \leq C_0 e^{-c_1 n^{\alpha}}$, $n=1,2,\dots$, then $\tau_n < \sqrt{2C_0} \gamma^{-1} \sqrt{n} e^{-c_2 n^\alpha}$, $n=1,2,\dots$, where $c_2 = 2^{-1-2\alpha}c_1$. The factor $\sqrt{n}$ can be removed by reducing the constant $c_2$. 
\end{lemma}

Once this generalization is done, by using inequality \eqref{conv::eq::interpError}, convergence rates for the interpolation error will easily follow (see section \ref{subsection::banach::ratesInterpolation}). 

%%%%%%%%%%%%%%%%%%%%%%%%%%%%%%%%%%%%%%%%%%%%%%%%%
%%%%%%%%% Convergence rates for $\tau_n$ %%%%%%%%
%%%%%%%%%%%%%%%%%%%%%%%%%%%%%%%%%%%%%%%%%%%%%%%%%

\subsection{Convergence rates for $(\tau_n)_{n=1}^\infty$ in the case where $(\gamma_n)_{n=1}^\infty$ is not constant}
\label{subsection::banach::ratesProjection}

We start by looking for an upper bound of the sequence $(\tau_n)_{n=1}^\infty$ that involves the sequence of Kolmogorov $n$-widths $(d_n)_{n=1}^\infty$. The case $n=1$ is addressed in lemma \ref{lemma::dim1}. The case $n>1$ is addressed in theorem \ref{theorem::banach::general}. From this last theorem, we infer corollaries \ref{corollary::banach::generalBis} and \ref{corollary::banach::general} that will be useful to derive convergence rates for $(\tau_n)_{n=1}^\infty$. 

%%%%%%%%%%%%%%%%% Dimension 1 %%%%%%%%%%%%%%%%%%%%%

\begin{lemma}
\label{lemma::dim1} For $n=1$, $\tau_1 \leq 2 (1+ \eta^{-1} ) d_1.$
\end{lemma}

\begin{proof}
Given the parameter $\eta$ in the GEIM greedy algorithm, let us chose $\beta > 1/\eta$. We begin by recalling and introducing some notations. First of all, $\varphi_0$ is the first interpolating function chosen in \eqref{eq:varphi0} by the greedy algorithm and $X_1 = \vspan \{\varphi_0 \}$. For any $\varphi$, $P_1(\varphi)$ is the metric projection of $\varphi$ onto $X_1$. Let $\varphi^{\max}_0 = \underset{\varphi \in F}{\argmax} \norm{\varphi}$. From \eqref{conv::eq::existenceMeshEqs} in the case $n=0$, $\norm{\varphi_0} \geq  \eta \norm{\varphi^{\max}_0} $.  Let $X_\mu$ be a one dimensional subspace and $E_\mu \coloneqq \underset{x \in F}{\max}\ \underset{y \in X_\mu}{\min} \ \norm{x-y}$. For any $\varphi \in F$, $\norm{\varphi - P_{X_\mu} (\varphi)} \leq E_\mu$. Let $\varphi^{\max}_\mu = \underset{\varphi\in F}{\argmax} \norm{ P_{X_\mu} (\varphi) }$. We now divide the proof by considering two cases of values of $ \norm{P_{X_\mu} (\varphi^{\max}_\mu)}$. If $ \norm{P_{X_\mu} (\varphi^{\max}_\mu)} \leq \frac{1+\eta}{\eta - 1/\beta}\ E_\mu$, for all $\varphi \in F$:
\begin{align*}
\norm{ \varphi - P_1(\varphi)} 
\leq  
\norm{\varphi}  
\leq  
\norm{\varphi-P_{X_\mu} (\varphi) } + \norm{ P_{X_\mu} (\varphi) }   
\leq  
E_\mu+ \norm{P_{X_\mu} (\varphi^{\max}_\mu)},
\end{align*}
which yields
\begin{equation}
\norm{ \varphi - P_1(\varphi)} 
\leq  
\left( 1+\frac{1+\eta}{\eta - 1/\beta} \right) E_\mu 
,\ \quad \forall \varphi \in F.
\end{equation}
If $\norm{P_{X_\mu} (\varphi^{\max}_\mu)} \geq \frac{1+\eta}{\eta - 1/\beta}\ E_\mu$, we have:
\begin{align}
\norm{P_{X_\mu} (\varphi_0) }    
    \geq  \norm{\varphi_0} -E_\mu 
    \geq  \eta \norm{\varphi^{\max}_0} - E_\mu 
    \geq  \eta \norm{\varphi^{\max}_\mu} - E_\mu,
\end{align}
and thus
\begin{align*}
\norm{P_{X_\mu} (\varphi_0) }     
\geq  \eta \left( \norm{P_{X_\mu} (\varphi^{\max}_\mu)} -E_\mu \right) - E_\mu 
\geq  \eta \norm{P_{X_\mu} (\varphi^{\max}_\mu)} -(1+\eta) E_\mu ,
\end{align*}
from which we infer that
\begin{equation}
\norm{P_{X_\mu} (\varphi_0) }    \geq  \norm{P_{X_\mu} (\varphi^{\max}_\mu) } /\beta. \label{eq::dim1_proj2}
\end{equation}
From inequality \eqref{eq::dim1_proj2}, it follows that $\norm{P_{X_\mu} (\varphi_0) }  >0 $ given that $\norm{P_{X_\mu} (\varphi^{\max}_\mu) }$ is  positive. Furthermore, for any $\varphi \in \mathcal{X}$, there exits $\lambda_{\varphi} \in \mathbb{R}$ such that: 
\begin{equation}
\label{conv::eq::dim1ineq1}
P_{X_\mu} (\varphi) = \lambda_{\varphi} P_{X_\mu}(\varphi_0).
\end{equation}
Hence the decomposition:
\begin{eqnarray}
\label{eq::dim1_decomposition}
\varphi&=& P_{X_\mu} (\varphi) + \varphi - P_{X_\mu} (\varphi)= 
\lambda_{\varphi} P_{X_\mu} (\varphi_0) + \varphi - P_{X_\mu} (\varphi)  \nonumber\\
 &=& \lambda_{\varphi}( P_{X_\mu} (\varphi_0) - \varphi_0 ) + \lambda_{\varphi} \varphi_0 + \varphi - P_{X_\mu}(\varphi).
\end{eqnarray}
Since $\norm{\varphi - P_1(\varphi)} \le   \norm{\varphi- \lambda_{\varphi} \varphi_0} $, we can use equation \eqref{eq::dim1_decomposition} to bound $\norm{\varphi- \lambda_{\varphi} \varphi_0} $ and write
\begin{align*}
\norm{\varphi - P_1(\varphi)} 
  \leq   |\lambda_{\varphi}| \norm{P_{X_\mu} (\varphi_0) - \varphi_0} + \norm{\varphi - P_{X_\mu}(\varphi)} 
  \leq  (1+ |\lambda_{\varphi}| ) E_\mu .
\end{align*}
Furthermore, given that $ \norm{P_{X_\mu} (\varphi^{\max}_\mu)} \geq \norm{P_{X_\mu} (\varphi)}$ for any $\varphi \in F$, we have
\begin{equation}
\label{eq::dim1_proj1}
\norm{P_{X_\mu} (\varphi^{\max}_\mu)} 
\geq  |\lambda_{\varphi}| \norm{ P_{X_\mu} (\varphi_0)},
\end{equation}
where we have used equality \eqref{conv::eq::dim1ineq1}. Inequalities \eqref{eq::dim1_proj2} and \eqref{eq::dim1_proj1} yield $| \lambda_{\varphi} | \leq \beta$ and therefore $\norm{ \varphi - P_1(\varphi) } \leq (1+\beta) E_\mu $. As a result, we have proven that for any $\beta > 1/\eta$ and any $\varphi \in F$, $\norm{ \varphi - P_1(\varphi) }  \leq  g_\eta (\beta) E_\mu $, where 
\begin{equation*}
\forall \beta >1/\eta,\quad g_\eta (\beta) := \max \left( 1+\beta ; 1+\frac{1+\eta}{\eta - 1/\beta} \right).
\end{equation*}
Thus $\norm{ \varphi - P_1(\varphi) } \leq  \underset{\beta >1/\eta}{\min} g_{\eta} (\beta)  E_\mu = 2(1+ \eta^{-1} ) E_\mu$. Since the inequality is valid for any one dimensional space $X_\mu$, the final result follows by taking the infimum over all one dimensional spaces in $\mathcal{X}$.
\end{proof}
%%%%%%%%%%%%%%%%% End Dimension 1 %%%%%%%%%%%%%%%%%%%%%
%%%%%%%%%%%%%%%%% Dimension n>1 %%%%%%%%%%%%%%%%%%%%%

%%%%%%%%%%%
\begin{theorem}
\label{theorem::banach::general}
For any $N \geq 0$, consider a weak greedy algorithm for which \eqref{eq:lemma::weakGreedy} holds. Then, for any $ K\geq 2,\ 1\leq m < K$ 
\begin{equation}
\label{conv::eq::mainThBanach}
\prod\limits_{i=1}^K \tau^2_{N+i} \leq  \dfrac{1}{\prod\limits_{i=1}^{K} \gamma^2_{N+i}} 2^K K^{K-m} \left( \sum\limits_{i=1}^{K}  \tau_{N+i}^2 \right)^m  d^{2(K-m)}_m .
\end{equation}
\end{theorem}
\begin{proof}
This result is an extension of theorem $4.1$ of \cite{devore2012greedy} to the case where the parameter $\gamma$ depends on $N$, the dimension of the space $X_N$. The proof is a slight modification of the one in \cite{devore2012greedy} but we provide it in the appendix for the self-consistency of this paper.
\end{proof}
%%%%%%%%%%%%%%%%%
\begin{corollary}
\label{corollary::banach::generalBis}
For any $n \geq 2$, 
\begin{equation}
\label{conv::eq::corolaryDeVore4.2}
\tau_n \leq \dfrac{1}{\prod\limits_{i=1}^{n} \gamma_{i}^{1/n} } \sqrt{2} \underset{ 1\leq m <n}{\min} \left\lbrace n^{\frac{n-m}{2n}} \left( \sum\limits_{i=1}^{n} \tau_i^2 \right)^{\frac{m}{2n}}  d_m^{\frac{n-m}{n}}  \right\rbrace
\end{equation}
In particular, for any $\ell \geq 1$: 
\begin{equation}
\label{conv::eq::tauNpairs}
\tau_{2\ell} \leq 2 \dfrac{1}{\prod\limits_{i=1}^{2 \ell} \gamma_{i}^{1/ 2\ell}} \sqrt{\ell d_{\ell}}.
\end{equation}
\end{corollary}

\begin{proof}
We take  $N=0$, $K=n$ and any $1\leq m < n$ in \eqref{conv::eq::mainThBanach} and use the monotonicity of $(\tau_n)_{n=1}^\infty$ to obtain:
\begin{equation*}
\tau_{n}^{2n} 
\leq 
\prod\limits_{i=1}^{n} \tau_i^2
\leq 
 \dfrac{1}{\prod\limits_{i=1}^{n} \gamma_{i}^2} 2^n n^{n-m} \left(  \sum\limits_{i=1}^{n} \tau_i^2 \right)^m d_m^{2(n-m)},
\end{equation*}
that yields \eqref{conv::eq::corolaryDeVore4.2}. In particular, if $n=2\ell$ and $m=\ell$, we have:
\begin{equation*}
\tau_{2\ell} 
\leq
\dfrac{1}{\prod\limits_{i=1}^{2 \ell} \gamma_{i}^{1/ 2\ell}} \sqrt{2} (2\ell)^{1/4} \left(   \sum\limits_{i=1}^{2\ell} \tau_i^2 \right)^{1/4} \sqrt{d_{\ell}}
\leq
\dfrac{1}{\prod\limits_{i=1}^{2 \ell} \gamma_{i}^{1/ 2\ell}} \sqrt{2} (2\ell)^{1/4} \left(   2\ell \right)^{1/4} \sqrt{d_{\ell}}
=
2 \dfrac{1}{\prod\limits_{i=1}^{2 \ell} \gamma_{i}^{1/ 2\ell}} \sqrt{\ell d_{\ell}},
\end{equation*}
where we used that all $\tau_i \leq 1$.
\end{proof}

%%%%%%%%%%%%%%%%%%%%%%%

\begin{corollary}
\label{corollary::banach::general}
For $N\geq 0$, $K \geq 2 $ and $1 \leq m < K$:
\begin{equation} 
\tau_{N+K} 
\leq  
\dfrac{1}{\prod\limits_{i=1}^{K} \gamma^{1/K}_{N+i}} \sqrt{2K} \tau^{m/K}_{N+1} d_m^{1-m/K}
\end{equation}
\end{corollary}
\begin{proof}
Using that $(\tau_n)_{n=1}^\infty$ is monotonically decreasing and following the same lines as above, we derive from inequality \eqref{conv::eq::mainThBanach} that:
\[
\tau_{N+K}^{2K} \leq \dfrac{1}{\prod\limits_{i=1}^{K} \gamma^2_{N+i}} 2^K K^{K-m}  \left( \sum\limits_{i=1}^{K}  \tau_{N+i}^2 \right)^m  d^{2(K-m)}_m.
\]
Therefore,
\[
\tau_{N+K} 
\leq  
\dfrac{1}{\prod\limits_{i=1}^{K} \gamma^{1/K}_{N+i}} \sqrt{2} K^{\frac{K-m}{2K}}   \left( K  \tau_{N+1}^2 \right)^{m/2K}  d^{1-m/K}_m  \\
\leq  
\dfrac{1}{\prod\limits_{i=1}^{K} \gamma^{1/K}_{N+i}} \sqrt{2K}  \tau_{N+1}^{m/K}  d^{1-m/K}_m .
\]
\end{proof}
%%%%%%%%%%%%%%%%%%

We now derive convergence rates for the sequence $(\tau_n)_{n=1}^\infty$. As we will see, the convergence and its rate strongly depend on the behavior of the Lebesgue constant. As discussed in section \ref{section::greedy}, $(\Lambda_n)_{n=1}^\infty$ can diverge exponentially. On the contrary and as it is often the case in practice, it can be polynomially increasing or even be bounded. We take this point into account by assuming different types of convergence rates for $(d_n)_{n=1}^\infty$. To begin with, Lemmas \ref{lemma::banach::polynomial} and \ref{lemma::banach::exponential} below consider the case where $(d_n)_{n=1}^\infty$ decreases polynomially or exponentially. No assumption on the behavior of $(\Lambda_n)_{n=1}^\infty$ is made in these results. With the reference to $\gamma_n$ introduced in \eqref{3K11}, we have
%%%%%%%%%% Polynomial decay Banach %%%%%%%%%%%%%%%%%%
\begin{lemma}
\label{lemma::banach::polynomial}
For any $n\geq 1$, let $n=4\ell+k$ (where $\ell \in \{0,1,\dots\}$ and $k \in \{0,1,2,3\}$). If, for $n\geq1,\ d_n \leq C_0 n^{-\alpha}$, 
with $C_0 >0$, then $\tau_n \leq  C_0 \beta_n n^{-\alpha}$, where  
\begin{equation*}
\beta_n = \beta_{4\ell +k} := 
\begin{cases}
2  \left( 1 + \eta^{-1} \right)    \quad\text{ if } n=1\\
\dfrac{1}{\prod\limits_{i=1}^{\ell_2} \gamma_{\ell_1 - \lceil \frac{k}{4} \rceil
 +i}^{\frac{1}{\ell_2}}}   \sqrt{2 \ell_2 \beta_{\ell_1}}  (2\sqrt{2})^\alpha  \quad  \text{if }  n \geq 2,
\end{cases} 
\end{equation*}
and $\ell_1 = 2\ell + \lfloor \frac{2k}{3} \rfloor$, $\ell_2 = 2\left( \ell + \lceil \frac{k}{4} \rceil\right)$, where $\lfloor \cdot \rfloor$ and $\lceil \cdot \rceil$  are the floor and ceiling functions respectively.
\end{lemma}
\begin{proof}
The proof is done by induction over $n$ and the case $n=1$ directly follows from lemma \ref{lemma::dim1}. In the case $n \geq 2$, we write $n=N+K$ with $N\geq 0$ and $K\geq 2$. Thanks to corollary \ref{corollary::banach::general}, we have for any $1\leq m <K$:
\begin{equation}\label{YY1}
\tau_n = \tau_{N+K} 
\leq  
\dfrac{1}{\prod\limits_{i=1}^{K} \gamma^{1/K}_{N+i}} \sqrt{2K} \tau^{m/K}_{N+1} d_m^{1-m/K}.
\end{equation}
We now use that $d_m \leq C_0 m^{-\alpha}$ and the recurrence hypothesis at index $N+1$ that states $\tau_{N+1} \leq C_0 \beta_{N+1} (N+1)^{-\alpha}$ which yields: 
\begin{equation}
\tau_{N+K}\leq C_0 \sqrt{2K} \dfrac{1}{\prod\limits_{i=1}^{K} \gamma_{N+i}^{\frac{1}{K}}} \beta_{N+1}^{\frac{m}{K}} \xi(N,K,m)^{\alpha} (N+K)^{-\alpha},
\end{equation}
where $\xi(N,K,m)= \dfrac{N+K}{m} \left( \dfrac{m}{N+1}\right)^{\frac{m}{K}} $ for any $1 \leq m < K$ and any given index $n=N+K \geq 2$, where $N \geq 0$ and $K \geq 2$. Furthermore, any $n\geq 2$ can be written as $n=4\ell+k$ with $\ell \in \mathbb{N}$ and $k \in \{0,1,2,3\}$. If $k=1,2$ or $3$, it can easily be proven that the function $\xi$ is bounded by $2\sqrt{2}$ by setting 
\begin{equation*}
\begin{cases}
N=2\ell -1 ,\ K=2\ell+2 ,\ m=\ell +1 \text{ and } \ell\geq 1  \text{ in the case } k=1,\\
N=2\ell ,\ K=2\ell+2 ,\ m=\ell+1 \text{ and } \ell\geq 0 \text{ in the case } k=2,\\
N=2\ell +1 ,\ K= 2\ell +2 ,\ m=\ell+1  \text{ and } \ell \geq 0 \text{ in the case } k=3.
\end{cases}
\end{equation*}
These choices of $N,\ K$ and $m$ combined with the upper bound of $\xi$ yield the result $\tau_n \leq C_0 \beta_n n^{-\alpha}$ in the case $k=1,2$ or $3$. To address the case $n=4\ell$, we come back to estimate \eqref{YY1} and use that $\tau_{N+1} \leq \tau_N$. It follows:
\begin{equation}
\label{eq::case_4l}
\tau_{n} \leq  \dfrac{1}{\prod\limits_{i=1}^{K} \gamma_{N+i}^{1/K}} \sqrt{2K} \tau^{m/K}_{N} d^{1-m/K}_m .
\end{equation} 
Choosing $N=K=2\ell$ and $m=\ell$, the inequality \eqref{eq::case_4l} directly yields the desired result:
\begin{equation*}
\tau_{4\ell} \leq C_0 \sqrt{2}\sqrt{2\ell\beta_{2\ell}} \dfrac{1}{\prod\limits_{i=1}^{2\ell} \gamma_{2\ell+i}^{\frac{1}{2\ell}}} (2\sqrt{2})^{\alpha} (4\ell)^{-\alpha}.
\end{equation*}
\end{proof}
%%%%%%%%%% Fin Polynomial decay Banach %%%%%%%%%%%%%%%%%%
%%%%%%%%%% Exponential decay Banach %%%%%%%%%%%%%%%%%%
\begin{lemma}
\label{lemma::banach::exponential}
If, for $n\geq 1$, $d_n \leq C_0 e^{-c_1 n^{\alpha}}$, with $C_0 \geq 1$ and $\alpha>0$, then $\tau_n \leq C_0\beta_n e^{-c_2 n^{\alpha}}, $
where $c_2 := c_1 2^{-2\alpha -1}$ and
\begin{equation*}
\beta_n := 
\begin{cases}
2  \left( 1+\eta^{-1}\right),\ \text{ if } n=1 \\
\sqrt{2}  \dfrac{1}{ \prod\limits_{i=1}^{2\lfloor \frac{n}{2}\rfloor} \gamma_i^{ \frac{1}{2\lfloor \frac{n}{2}\rfloor}  } } \sqrt{n},\ \text{ if } n\geq 2.
\end{cases}
\end{equation*}
\end{lemma}

\begin{proof}
The case $n=1$ easily follows from lemma \ref{lemma::dim1}. For $n=2\ell$ ($\ell\geq 1$), inequality \eqref{conv::eq::tauNpairs} directly yields:
\begin{equation}
\label{conv::eq::banachExpoPair}
\tau_{2\ell} 
\leq  
2 \dfrac{1}{\prod\limits_{i=1}^{2\ell} \gamma_i^{1/2\ell} } \sqrt{\ell d_{\ell}} 
\leq 
C_0 \sqrt{2} \dfrac{1}{\prod\limits_{i=1}^{2\ell} \gamma_i^{1/2\ell} } \sqrt{2\ell} e^{ -\frac{c_1}{ 2^{1+\alpha}} (2\ell)^{\alpha}  },
\end{equation}
where we used that $d_{\ell} \leq C_0 e^{-c_1 (\ell)^{\alpha}}$ and that $C_0 \geq 1$. For $n=2\ell+1$, by using inequality \eqref{conv::eq::banachExpoPair} and $\tau_{2\ell+1} \leq \tau_{2\ell}$, we have:
\begin{equation}
\label{conv::eq::banachExpoImpair}
\tau_{2\ell+1} 
\leq 
C_0\sqrt{2} \dfrac{1}{\prod\limits_{i=1}^{2\ell} \gamma_i^{1/2\ell} } \sqrt{2\ell} e^{ -\frac{c_1}{ 2^{1+\alpha}} (2\ell)^{\alpha}  }
\leq 
C_0\sqrt{2} \dfrac{1}{\prod\limits_{i=1}^{2\ell} \gamma_i^{1/2\ell} } \sqrt{2\ell+1} e^{ -\frac{c_1}{ 2^{1+2\alpha}} (2\ell+1)^{\alpha}  }
.
\end{equation}
\end{proof}
%%%%%%%%%%%%%%%%%%%%%%%%%%%%%%%%%%%%%%

From lemmas \ref{lemma::banach::polynomial} and \ref{lemma::banach::exponential} we observe that, if $(\Lambda_n)_{n=1}^\infty$ diverges exponentially, then an exponential decay is required for $(d_n)_{n=1}^\infty$. Let us now derive some results by adding different assumptions in the behavior of $(\Lambda_n)_{n=1}^\infty$. In corollary \ref{corollary::banach::lambdaIncreases}, we assume that $(\Lambda_n)_{n=1}^\infty$ is monotonically increasing (i.e., $(\gamma_n)_{n=1}^\infty$ monotonically decreasing). 

%%%%%%%%%%%%%%%%%%%%%
\begin{corollary}
\label{corollary::banach::lambdaIncreases}
Assume that $(\Lambda_n)_{n=1}^\infty$ is monotonically increasing, then:
\begin{itemize}
\item[i)] If $d_n \leq C_0 n^{-\alpha}$ for any $n\geq 1$,
then $\tau_n \leq C_0 \tilde{\beta}_n n^{-\alpha}$, with 
$$ \tilde{\beta}_n :=  
\begin{cases}
2  \left( 1+\eta^{-1} \right),\ &\text{if } n=1 \\
 2^{3\alpha +1} \ell_2 \gamma_n^{-2},\ &\text{if } n\geq 2 .
\end{cases}
$$
If we write $n$ as $n=4\ell+k$ (with $\ell\in\{0,1,\dots\}$ and $k\in\{0,1,2,3\}$), then $\ell_2 = 2\left( \ell + \lceil \frac{k}{4} \rceil\right)$.
\item[ii)] If $d_n \leq C_0 e^{-c_1 n^{\alpha}}$ for
$n\geq 1$ and $C_0\geq 1$, then $\tau_n \leq C_0 \tilde{\beta}_n e^{-c_2
n^{-\alpha}}$, with $c_2 = c_1 2^{-2\alpha -1}$ and
\begin{equation*}
\tilde{\beta}_n :=  
\begin{cases}
2  \left( 1+\eta^{-1} \right),\ &\text{ if } n=1 \\
\sqrt{2n}  \gamma_n^{-1} ,\ &\text{ if } n\geq 2.
\end{cases}
\end{equation*}
\end{itemize}
\end{corollary}

\begin{proof}~\\

\begin{itemize}
\item[i)] We show by induction that $\tilde{\beta}_n $ is larger than the coefficient $\beta_n$ defined in lemma \ref{lemma::banach::polynomial}. If $n=1$, $\tilde{\beta_1} = \beta_1$. Then, for $n >1$, given that $(\gamma_n)_{n=1}^\infty$ is monotonically decreasing, 
$$
\beta_n 
\leq 
\gamma_n^{-1} \sqrt{2\ell_2 \beta_{\ell_1}} \left(  2\sqrt{2} \right)^{\alpha}
\leq
\gamma_n^{-1} \sqrt{2\ell_2 \tilde{\beta}_{\ell_1}} \left(  2\sqrt{2} \right)^{\alpha},
$$
where we used the recurrence hypothesis $\beta_{\ell_1} \leq \tilde{\beta}_{\ell_1}$ in the second inequality. Furthermore, since $\tilde{\beta}_{\ell_1} \leq  2^{3\alpha+1} \ell_2 \gamma_n^{-2}$, it follows that:
$$
\beta_n 
\leq 
\gamma_n^{-1} \sqrt{2\ell_2 2^{3\alpha+1}\ell_2\gamma_n^{-2}} \left(  2\sqrt{2} \right)^{\alpha} 
= 
2^{3\alpha +1} \ell_2 \gamma^{-2}_n = \tilde{\beta}_n.
$$
\item[ii)] The result is straightforward and follows from the definition of $\beta_n$ given in lemma \ref{lemma::banach::exponential}. 
\end{itemize}
\end{proof}

If $(\Lambda_n)_{n=1}^\infty$ is constant, corollary \ref{corollary::banach::lambdaIncreases} shows that we obtain exactly the same result as the one derived in \cite{devore2012greedy} for the exponential case (recalled in lemma \ref{conv::lemma:deVoreBanachExp}). In the case of polynomial decay, the result of corollary \ref{corollary::banach::lambdaIncreases} provides a slightly degraded result with respect to the one in 
\cite{devore2012greedy} (recalled in lemma \ref{conv::lemma:deVoreBanachPoly}). The most important difference relies on the fact that in~\cite{devore2012greedy} a convergence rate of order $\ord( n^{-\alpha+1/2+\varepsilon})$ is obtained whereas the present results yields a convergence in $\ord( n^{-\alpha+1})$. It has so far not been possible to derive better convergence rates in the polynomial case for a general behavior of the sequence $(\Lambda_n)_{n=1}^\infty$. However, under the refined assumption 
$$
\Lambda_n = \ord (n^{\zeta}),
$$
which is a typical behavior observed in numerical applications, Lemma \ref{lemma::banach::polynomial2} below shows that, in this case, the convergence is of order $\ord ( n^{-\alpha+ \zeta +1/2+\varepsilon})$, which is consistent with the result of \cite{devore2012greedy}, in case of a constant $\Lambda_n$ ($\zeta=0$).

%%%%%%%%%%%%%%%%%%%%%%%%%%%%%%%%%%%%%%%%%%%%%%%%%%%%%
%%%%%% Polynomial decay Banach with Lebesgue %%%%%%%% %%%%%%       linearly incresing              %%%%%%%%
%%%%%%%%%%%%%%%%%%%%%%%%%%%%%%%%%%%%%%%%%%%%%%%%%%%%%

\begin{lemma}
\label{lemma::banach::polynomial2}
If, for $n>0$, $d_n \leq C_0 n^{-\alpha}$ and $\gamma_n^{-1} \leq C_{\zeta} n^{\zeta}$, with constants $C_0,\ C_{\zeta},\ \alpha,\ \zeta>0$, then for any $\beta>1/2$, we have $\tau_n\leq C_1 n^{-\alpha+\zeta+\beta}$, $n\in \mathbb{N}^{*}$, where
\begin{equation*}
C_1 := \max \left\lbrace  C_0 2^{\frac{2\alpha^2}{\zeta}} \left( \dfrac{\zeta+\beta}{\beta-\frac{1}{2}} \right)^{\alpha} \max\left( 1; C_{\zeta}^{\frac{\zeta+\beta}{\zeta}}   \right)    
;\ 
\underset{n=1,\dots, 2 \lfloor 2(\zeta +\beta) \rfloor +1}{\max} n^{\alpha-\zeta-\beta}
 \right\rbrace.
\end{equation*}

\end{lemma}
Note that in the above lemma, the constant $\beta$ has no connection with $\beta_n$ defined in section \ref{section::greedy}.

\begin{proof}
It follows from the monotonicity of $(\tau_n)_{n=1}^\infty$ and inequality \eqref{conv::eq::mainThBanach} for $N=K=n$ and any $1\leq m <n$ that:
\begin{equation}
\label{conv::eq::polyDecay2eq1}
\tau_{2n} \leq \sqrt{2n} \dfrac{1}{\prod\limits_{i=1}^{n} \gamma_{n+i}^{1/n}} \tau_n^\delta d_m^{1-\delta}, \quad \delta:=\dfrac{m}{n}.
\end{equation}
Given $\beta>1/2$, we define 
$m:=\Bigl\lfloor \dfrac{\beta-\frac{1}{2}}{\zeta + \beta}   \Bigr\rfloor +1$ (so that $m<n$ for $n>2(\zeta+\beta)>2\zeta+1$). It follows that
\begin{equation}
\label{conv::eq::polyDecay2eqDelta}
\delta = \dfrac{m}{n} \in \left(   \dfrac{\beta-\frac{1}{2}}{\zeta+\beta} , \dfrac{\beta-\frac{1}{2}}{\zeta+\beta} + \dfrac{1}{n}  \right).
\end{equation}
We prove our claim by contradiction. Suppose it is not true and $M$ is the first value where $\tau_M > C_1 M^{-\alpha+\zeta+\beta}$. Clearly, because of the definition of $C_1$ and the fact that $\tau_n \leq 1$, we must have $M > 2\lfloor 2(\zeta +\beta) \rfloor +1$ (since $M \geq 2 \lfloor 2(\zeta +\beta) \rfloor +2$). We first consider the case $M=2n$, and therefore $n\geq \lfloor 2(\zeta +\beta) \rfloor +1$. From \eqref{conv::eq::polyDecay2eq1}, we have:
\begin{eqnarray}
C_1 (2n)^{-\alpha+\zeta+\beta} 
<
\tau_{2n}
&\leq &
\sqrt{2n} \dfrac{1}{\prod\limits_{i=1}^{n} \gamma_{n+i}^{1/n} } \tau_n^{\delta} d_m^{1-\delta} \nonumber\\
&\leq &
\sqrt{2n} C_{\zeta} (2n)^{\zeta} C_1^{\delta} n^{\delta (-\alpha + \zeta+\beta)} C_0^{1-\delta} (\delta n)^{-\alpha (1-\delta)}, \label{conv::eq::polyDecay2eq1bis}
\end{eqnarray}
where we have used the fact that $\tau_n \leq C_1 n^{-\alpha+\zeta+\beta}$ and $d_m \leq C_0 m^{-\alpha}$. It follows that
\begin{equation*}
C_1^{1-\delta} < 2^{\alpha - \beta +\frac{1}{2}} C_{\zeta} C_0^{1-\delta} \delta^{-\alpha (1-\delta)} n^{   \frac{1}{2}+\delta (\zeta+\beta)-\beta }
\end{equation*}
and therefore
\begin{equation*}
C_1 < 2^{ \frac{\alpha - \beta +\frac{1}{2} }{1-\delta}  }  C_{\zeta}^{\frac{1}{1-\delta}} C_0 \delta^{-\alpha} n^{  \frac{\zeta+\beta}{1-\delta} \left(  \delta - \frac{\beta-\frac{1}{2}}{\zeta+\beta} \right) }.
\end{equation*}
Since, for $n\geq \lfloor 2(\zeta +\beta) \rfloor +1 >2(\zeta+\beta)$, we have
\begin{equation}
\delta  <  \dfrac{\beta - \frac{1}{2}}{\zeta + \beta} + \dfrac{1}{n} <  \dfrac{\beta}{\zeta +\beta},   \label{conv::eq::polyDecay2eq2}
\end{equation}
then,
\begin{equation}
\label{conv::eq::polyDecay2eq3}
\dfrac{1}{1-\delta} < \dfrac{\zeta + \beta}{ \zeta}.
\end{equation}
Hence,
\begin{equation}
\label{conv::eq::polyDecay2eq4}
\dfrac{\zeta+\beta}{1-\delta} \left(  \delta - \dfrac{\beta - \frac{1}{2}}{\zeta + \beta}  \right)
< 
\left( \dfrac{\zeta +\beta}{1-\delta} \right)\dfrac{1}{n}
<
\dfrac{(\zeta + \beta)^2}{\zeta} \dfrac{1}{n},
\end{equation}
where we have used inequalities \eqref{conv::eq::polyDecay2eq2} and \eqref{conv::eq::polyDecay2eq3}. By using \eqref{conv::eq::polyDecay2eq4}, it follows that
\begin{equation}
n^{ \frac{\zeta+\beta}{1-\delta} \left(  \delta - \frac{\beta - \frac{1}{2}}{\zeta + \beta}  \right) }
<
n^{  \frac{(\zeta + \beta)^2}{\zeta} \frac{1}{n}  }
<
2^{  \frac{(\zeta+\beta)^2}{\zeta}  }.
\end{equation}
This yields:
\begin{equation}
\label{conv::eq::polyDecay2eq5}
C_1 
< 
2^{ \frac{\alpha - \beta +\frac{1}{2} }{1-\delta}  }  C_{\zeta}^{\frac{1}{1-\delta}} C_0 \delta^{-\alpha} 2^{  \frac{(\zeta+\beta)^2}{\zeta}  }
<
2^{ \left( \frac{\zeta+\beta}{\zeta}  \right) (\alpha + \zeta + \frac{1}{2}) } C_{\zeta}^{\frac{1}{1-\delta}} C_0 \delta^{-\alpha}
.
\end{equation}
Furthermore, for $-\alpha + \zeta + \beta < 0$ (which is the meaningful case), and using the fact that $\beta>\dfrac{1}{2}$, we have:
\begin{equation}
\label{conv::eq::polyDecay2eq6}
2^{ \frac{\zeta + \beta}{ \zeta} \left( \alpha +\zeta + \frac{1}{2} \right)   }
<
2^{ \frac{\alpha}{ \zeta} \left( \alpha +\zeta + \beta \right)   }
<
2^{ \frac{2\alpha^2}{\zeta}   } \quad \hbox{and} \quad 
C_{\zeta} ^{ \frac{1}{1-\delta}  } < \max\left(  1; C_{\zeta}^{ \frac{\zeta +\beta}{ \zeta} }   \right).
\end{equation}
Also, from \eqref{conv::eq::polyDecay2eqDelta}, we have
\begin{equation}
\label{conv::eq::polyDecay2eq8}
\delta^{-\alpha} < \left(  \dfrac{\zeta + \beta}{\beta - \frac{1}{2}}  \right)^{\alpha}.
\end{equation}
By inserting inequalities \eqref{conv::eq::polyDecay2eq6} and \eqref{conv::eq::polyDecay2eq8} in \eqref{conv::eq::polyDecay2eq5}, the desired contradiction follows:
\begin{equation*}
C_1 < C_0 2^{\frac{2\alpha^2}{\zeta}} \left( \dfrac{\zeta+\beta}{\beta-\frac{1}{2}} \right)^{\alpha} \max\left( 1; C_{\zeta}^{\frac{\zeta+\beta}{\zeta}}   \right).
\end{equation*}
Likewise, if $M=2n+1$, then  $M \geq 2 \lfloor 2(\zeta +\beta) \rfloor +3$, which implies that $n\geq \lfloor 2(\zeta +\beta) \rfloor +1$:
\begin{eqnarray}
C_1 2^{-\alpha + \zeta +\beta} (2n)^{-\alpha + \zeta +\beta} 
<
C_1 (2n+1)^{-\alpha + \zeta + \beta}
<
\tau_{2n+1}
\leq 
\tau_{2n}.
\end{eqnarray}
But, since we have from equation \eqref{conv::eq::polyDecay2eq1bis}
\begin{equation}
\tau_{2n}
\leq
\sqrt{2n} C_{\zeta} (2n)^{\zeta} C_1^{\delta} n^{\delta (-\alpha + \zeta+\beta)} C_0^{1-\delta} (\delta n)^{-\alpha (1-\delta)},
\end{equation}
then, following the same argument as above, we get:
\begin{equation*}
C_1 
< 
C_0 2^{  \left( \frac{\zeta+\beta}{\zeta}  \right) \left( \frac{1}{2} + 2\alpha -\beta   \right)  } \left( \dfrac{\zeta+\beta}{\beta-\frac{1}{2}} \right)^{\alpha} \max\left( 1; C_{\zeta}^{\frac{\zeta+\beta}{\zeta}}   \right) 
\nonumber
<
C_0 2^{\frac{2\alpha^2}{\zeta}} \left( \dfrac{\zeta+\beta}{\beta-\frac{1}{2}} \right)^{\alpha} \max\left( 1; C_{\zeta}^{\frac{\zeta+\beta}{\zeta}}   \right),
\end{equation*}
where we have used the fact that $\beta> 1/2$ in the last inequality.
\end{proof}

%%%%%%%%%%%%%%%%%%%%%%%%%%%%%%%%%%%%%%%%%%%%%%%%%%%%%
%%%%%% Polynomial decay Banach with Lebesgue %%%%%%%% %%%%%%       linearly incresing              %%%%%%%%
%%%%%%%%%%%%%%%%%%%%%%%%%%%%%%%%%%%%%%%%%%%%%%%%%%%%%

%%%%%%%%%%%%%%%%%%%%%%%%%%%%%%%%%%%%%%%%%%%%%%%%%%%%%%
%%%%%%%%%%%% Convergence rates EIM Banach %%%%%%%%%%%%
%%%%%%%%%%%%%%%%%%%%%%%%%%%%%%%%%%%%%%%%%%%%%%%%%%%%%%

\subsection{Convergence rates of the interpolation error}
\label{subsection::banach::ratesInterpolation}
Thanks to the convergence rates obtained for $(\tau_n)_{n=1}^\infty$ in section \ref{subsection::banach::ratesProjection}, the following rates are readily obtained for the interpolation error of GEIM. 
\begin{theorem}~\\
\vspace*{-0.5cm}
\begin{itemize}
\item[i)] If $d_n \leq C_0 n^{-\alpha}$ for any $n\geq 1$, then $\Vert \varphi-{\cal J}_n[\varphi] \Vert _{\mathcal{X}} \leq  (1+\Lambda_n)  C_0 \beta_n n^{-\alpha}$ for any $\varphi \in F$, where $\beta_n$ is given in lemma \ref{lemma::banach::polynomial}.
\item[ii)] If $d_n \leq C_0 e^{-c_1 n^{\alpha}}$ for  $n\geq 1$ and $C_0\geq 1$, $\Vert \varphi-{\cal J}_n[\varphi] \Vert _{\mathcal{X}} \leq  (1+\Lambda_n) C_0 \beta_n e^{-c_2 n^{\alpha}}$ for any $\varphi \in F$, where $\beta_n$ and $c_2$ are defined as in lemma \ref{lemma::banach::exponential}.
\end{itemize}
\end{theorem}
\begin{proof}
It follows from equation \eqref{conv::eq::interpError} and the definition of $\tau_n$ that, $\forall \varphi \in F,\ \Vert \varphi-{\cal J}_n[\varphi] \Vert _{\mathcal{X}} \leq (1+\Lambda_n) \Vert \varphi -P_n(\varphi) \Vert _{\mathcal{X}}  \leq (1+\Lambda_n) \tau_n$. We conclude the proof by bounding $\tau_n$ by using lemmas \ref{lemma::banach::polynomial} and \ref{lemma::banach::exponential}.
\end{proof}
%%%%%%%%%%%%%%%%%%%%%%%%%%%%%%%%%%%%%%%%%%

\noindent
If $(\Lambda_n)_{n=1}^\infty$ is a monotonically increasing sequence, we have a more precise behavior:
\begin{corollary}
\label{conv::corollary::banachLambdaMonotIncreasing}
Let $(\Lambda_n)_{n=1}^\infty$ be a monotonically increasing sequence. Then,
\begin{itemize}
\item[i)] if $d_n \leq C_0
n^{-\alpha}$ for any $n\geq 1$, then
$$
\forall \varphi \in F, \quad
\Vert \varphi-{\cal
J}_n[\varphi] \Vert _{\mathcal{X}} \leq 
\begin{cases}
2C_0 \left( 1+ \eta^{-1}\right) (1+\Lambda_1),\quad &\text{if } n=1. \\
C_0 2^{3\alpha +1} \ell_2 (1+\Lambda_n)^3 \eta^{-2}  n^{-\alpha},\quad &\text{if } n\geq 2.
\end{cases}
$$
If we write $n$ as $n=4\ell+k$ (with $\ell\in\{0,1,\dots\}$ and $k\in\{0,1,2,3\}$), then $\ell_2 = 2\left( \ell + \lceil \frac{k}{4} \rceil\right)$.
\item[ii)] if $d_n \leq C_0 e^{-c_1
n^{\alpha}}$ for $n\geq 1$ and $C_0 \geq 1$, then (remember $c_2 = c_1 2^{-2\alpha-1}$)
$$
\forall \varphi \in F, \quad
\Vert \varphi-{\cal
J}_n[\varphi] \Vert _{\mathcal{X}} \leq 
\begin{cases}
2C_0 \left( 1+ \eta^{-1} \right) (1+\Lambda_1),\quad &\text{if } n=1, \\
C_0 \sqrt{2}  (1+\Lambda_n)^2\eta^{-1} \sqrt{n} e^{-c_2 n^{\alpha}},\quad &\text{if } n\geq 2,
\end{cases}
$$
\item[iii)] if $d_n \leq C_0 n^{-\alpha}$ and $\gamma_n^{-1} \leq C_{\zeta} n^{\zeta}$ for any $n\geq 1$, then for any $\beta>1/2$,
$$\forall\varphi \in F,\quad \Vert \varphi-{\cal J}_n[\varphi] \Vert _{\mathcal{X}} \leq \eta C_{\zeta} C_1 n^{-\alpha + 2\zeta +\beta},$$
where the parameter $C_1$ is defined as in lemma \ref{lemma::banach::polynomial2}.
\end{itemize}
\end{corollary}

\begin{proof}
$i)$ and $ii)$ easily follow from corollary \ref{corollary::banach::lambdaIncreases} and $iii)$ is derived by using lemma \ref{lemma::banach::polynomial2}.
\end{proof}

%%%%%%%%%%%%%%%%%%%%%%%%%%%%%%%%%%%%%%%%%%%%%%%%%%%%%%%%%%%

\section{Convergence rates of GEIM in a Hilbert space}
\label{section::hilbert}

% \subsection{Preliminary notations and properties}

In this section, $\mathcal{X}$ is a Hilbert space equipped with its induced norm $\norm{f} = \ps{f}{f}$, where $\ps{.}{.}$ is the scalar product in $\mathcal{X}$. In the same spirit as in the case of a Banach space, we define the sequence $(\tau_n)_{n=1}^\infty$ as in formula \eqref{eq::tau} but now, for any $f \in F$, $P_n(f)$ corresponds to the unique element of $X_n$ that is the orthogonal projection of $f$ onto $X_n$.  Note that lemma \ref{lemma::weakGreedy} still holds in the Hilbert setting. We derive convergence rates for the interpolation error by applying the same strategy as in the Banach space case. In section \ref{subsection::hilbert::ratesProjection}, we derive convergence rates for  $(\tau_n)_{n=1}^\infty$ as an intermediate step. We compare to \cite{devore2012greedy} in corollary $3.3$ by taking $\gamma_n = \gamma$ in our results. The results of \cite{devore2012greedy} read:
\begin{lemma}[Corollary $3.3-(ii)$ of \cite{devore2012greedy}]
\label{conv::lemma:deVoreHilbertPoly}
If $d_n \leq C_0 n^{-\alpha}$ for $n= 1,2,\dots$, then $\tau_n \leq C_1 n^{-\alpha}$, $n= 1,2,\dots$, with $C_1 = 2^{5\alpha +1} \gamma^{-2} C_0$.
\end{lemma}
\begin{lemma}[Corollary $3.3-(iii)$ of \cite{devore2012greedy}]
\label{conv::lemma:deVoreHilbertExp}
If $d_n \leq C_0 e^{-c_1 n^{\alpha}}$ for $n= 1,2,\dots$, then $\tau_n < \sqrt{2C_0} \gamma^{-1} e^{-c_2 n^\alpha}$, $n= 1,2,\dots$, where $c_2=2^{-1-2\alpha}c_1$.
\end{lemma}

\subsection{Convergence rates for $(\tau_n)_{n=1}^\infty$}
\label{subsection::hilbert::ratesProjection}
Like in the Banach space case, we start by bounding the sequence $(\tau_n)_{n=1}^\infty$ with respect to $(d_n)_{n=1}^\infty$. This is done in theorem \ref{theorem::hilbert} (the analogue of theorem \ref{theorem::banach::general}). It yields corollaries \ref{corollary::hilbert::generalBis} and \ref{corollary::hilbert::general}, that are the analogue of corollaries \ref{corollary::banach::generalBis} and \ref{corollary::banach::general}. The major difference with respect to the Banach space case is the absence of a factor $\sqrt{n}$ in corollaries  \ref{corollary::hilbert::generalBis} and \ref{corollary::hilbert::general}. It will be the key to obtain improved results in Hilbert spaces. 
\begin{theorem}
\label{theorem::hilbert}
For any $N \geq 0$, consider a weak greedy algorithm for which \eqref{eq:lemma::weakGreedy} holds. Then, for any $ K\geq 2,\ 1\leq m < K$,
\[  \prod\limits_{i=1}^K \tau^2_{N+i} \leq  \dfrac{1}{\prod\limits_{i=1}^{K} \gamma^2_{N+i}} \left( \dfrac{K}{m}\right)^m \left( \dfrac{K}{K-m}\right)^{K-m} \tau^{2m}_{N+1} d^{2(K-m)}_m .\]
\end{theorem}
\begin{proof}
See appendix B.
\end{proof}
\begin{corollary}
\label{corollary::hilbert::generalBis}
For $n\geq 2$, 
\begin{equation}
\tau_n \leq \sqrt{2} \dfrac{1}{ \prod\limits_{i=1}^{n} \gamma_i^{1/n}} \underset{1\leq m < n}{\min} d_m^{\frac{n-m}{n}}.
\end{equation}
In particular, for any $\ell\ge 1$
\begin{equation}
\tau_{2\ell} \leq \sqrt{2} \dfrac{1}{\prod\limits_{i=1}^{2\ell} \gamma_i^{\frac{1}{2\ell}}} \sqrt{d_\ell}.
\end{equation}
\end{corollary}
\begin{corollary}
\label{corollary::hilbert::general}
For $N\geq 0$, $K \geq 2 $ and $1 \leq m < K$:
\begin{equation} 
\tau_{N+K} 
\leq  
\dfrac{1}{\prod\limits_{i=1}^{K} \gamma^{1/K}_{N+i}} \sqrt{2} \tau^{m/K}_{N+1} d_m^{1-m/K}.
\end{equation}
\end{corollary}
\begin{proof}
The proofs of corollaries \ref{corollary::hilbert::generalBis} and \ref{corollary::hilbert::general}
follow very similar guidelines as the ones for corollaries  \ref{corollary::banach::generalBis} and \ref{corollary::banach::general}. The only difference is that here the staring point is theorem \ref{theorem::hilbert} instead of \ref{theorem::banach::general}.
\end{proof}

Using theorem \ref{theorem::hilbert}, we derive decay rates of the sequence $(\tau_n)_{n=1}^\infty$ when $(d_n)_{n=1}^\infty$ has a polynomial or an exponential decay. In lemmas \ref{lemma::hilbert::polynomial} and \ref{lemma::hilbert::exponential}, no assumption on the behavior of $(\Lambda_n)_{n=1}^\infty$ is made.

%%%%%%%% Polynomial decay Hilbert %%%%%%%%%%%%%
\begin{lemma}
\label{lemma::hilbert::polynomial}
For any $n\geq 1$, let $n=4\ell+k$ (where $\ell \in \{0,1,\dots\}$ and $k \in \{0,1,2,3\}$).
If, for $n\geq 1$, $d_n \leq C_0 n^{-\alpha}$ with $C_0>0$, then  $\tau_n \leq C_0 \beta_n n^{-\alpha}$, where 
\begin{equation*}
\beta_n = \beta_{4\ell +k} := 
\begin{cases}
2  \left( 1+ \eta^{-1}\right)   &\quad\text{if } n=1\\
\sqrt{2\beta_{\ell_1}} \dfrac{1}{\prod\limits_{i=1}^{\ell_2} \gamma_{\ell_1 - \lceil \frac{k}{4} \rceil
 +i}^{\frac{1}{\ell_2}}} (2\sqrt{2})^\alpha  &\quad  \text{if }  n \geq 2
\end{cases}
\end{equation*}
and $\ell_1 = 2\ell + \lfloor \frac{2k}{3} \rfloor$, $\ell_2 = 2\left( \ell + \lceil \frac{k}{4} \rceil\right)$.
\end{lemma}

\begin{proof}
The proof is similar with the one of lemma \ref{lemma::banach::polynomial}: the case $n=1$ directly follows from lemma \ref{lemma::dim1} and if $n \geq 2$, we write $n=N+K$ with $N\geq 0$ and $K\geq 2$. Corollary \ref{corollary::hilbert::general} yields
\begin{equation*}
\tau_{N+K} 
\leq  
\dfrac{1}{\prod\limits_{i=1}^{K} \gamma^{1/K}_{N+i}} \sqrt{2} \tau^{m/K}_{N+1} d_m^{1-m/K}.
\end{equation*}
By using that $d_m \leq C_0 m^{-\alpha}$ and the recurrence hypothesis $\tau_{N+1} \leq \beta_{N+1} (N+1)^{-\alpha}$, it follows that 
\begin{equation*}
\tau_{N+K}\leq C_0 \sqrt{2} \dfrac{1}{\prod\limits_{i=1}^{K} \gamma_{N+i}^{\frac{1}{K}}} \beta_{N+1}^{\frac{m}{K}} \xi(N,K,m)^{\alpha} (N+K)^{-\alpha},
\end{equation*}
where $\xi(N,K,m)= \dfrac{N+K}{m} \left( \dfrac{m}{N+1}\right)^{\frac{m}{K}} $ for any $1 \leq m < K$ and any given index $n=N+K \geq 2$, where $N \geq 0$ and $K \geq 2$. It suffices now to decompose any $n\geq 2$ as $n=4\ell+k$ with $\ell \in \{0,1,\dots\}$ and $k \in \{0,1,2,3\}$ and use the same choices of $N,\ K$ and $m$ described in the proof of lemma \ref{lemma::banach::polynomial} to derive the result.
\end{proof}

%%%%%%%% End Polynomial decay Hilbert %%%%%%%%%%%%%

%%%%%%%% Exponential decay Hilbert %%%%%%%%%%%%%

\begin{lemma}
\label{lemma::hilbert::exponential}
If, for $n\geq 1$, $d_n \leq C_0 e^{-c_1 n^{\alpha}}$ with $C_0 \geq 1$, then 
$\tau_n \leq C_0 \beta_n e^{-c_2 n^{\alpha}}$,
where $c_2 := c_1 2^{-2\alpha -1}$ and
$\beta_1=2  ( 1+\eta^{-1})$, $\beta_n := \sqrt{2}  \dfrac{1}{ \prod\limits_{i=1}^{2\lfloor \frac{n}{2}\rfloor} \gamma_i^{ \frac{1}{2\lfloor \frac{n}{2}\rfloor}  } }$ for $n\geq 2$.
\end{lemma}

\begin{proof}
The proof is the same as lemma \ref{lemma::banach::exponential} but uses corollary \ref{corollary::hilbert::generalBis} instead of corollary \ref{corollary::banach::generalBis}.
\end{proof}

%%%%%%%% End exponential decay Hilbert %%%%%%%%%%%%%

As in the case of Banach spaces, it is important to study convergence rates in the case where $(\Lambda_n)_{n=1}^\infty$ is monotonically increasing. The  following corollary accounts for it.
\begin{corollary}
\label{corollary::hilbert::lambdaIncreases}
If $(\Lambda_n)_{n=1}^\infty$ is a monotonically increasing sequence then
\begin{itemize}
\item[i)] if $d_n \leq C_0 n^{-\alpha}$ for any $n\geq 1$,
then $\tau_n \leq C_0 \tilde{\beta}_n n^{-\alpha}$, with 
$$ \tilde{\beta}_n :=  
\begin{cases}
2  \left( 1+\eta^{-1} \right),\ &\text{ if } n=1 \\
 2^{3\alpha +1} \gamma_n^{-2},\ &\text{ if } n\geq 2.
\end{cases}
$$
\item[ii)] if $d_n \leq C_0 e^{-c_1 n^{\alpha}}$ for
$n\geq 1$ and $C_0\geq 1$, then $\tau_n \leq C_0 \tilde{\beta}_n e^{-c_2
n^{-\alpha}}$, with 
$$ \tilde{\beta}_n :=  
\begin{cases}
2 \left( 1+\eta^{-1} \right),\ &\text{ if } n=1 \\
\sqrt{2}  \dfrac{1}{\gamma_n },\ &\text{ if } n\geq 2.
\end{cases}
$$
\end{itemize}
\end{corollary}

\begin{proof}
The proof is derived by following the same guidelines as the proof of corollary \ref{corollary::banach::lambdaIncreases}.
\end{proof}
As a direct consequence of corollary \ref{corollary::hilbert::lambdaIncreases}, if $\gamma_n$ is constant, we recover slightly better results than the ones  in \cite{devore2012greedy} for $n\geq 2$ (see lemmas \ref{conv::lemma:deVoreHilbertPoly} and \ref{conv::lemma:deVoreHilbertExp} above).

%%%%%%%%%%%%%%%%%%%%%%%%%%%%%%%%%%%%%%%%%

\subsection{Convergence rates of the interpolation error}
Following similar guidelines as in the case of Banach spaces, the following rates can easily be derived for the interpolation error of GEIM.
\begin{theorem}~\\
\label{theorem::hilbert::convergenceGEIM}
\vspace*{-0.5cm}
\begin{enumerate}
\item If, for $n\geq 1$, $d_n \leq C_0 n^{-\alpha}$, with $C_0 > 0$, then$\Vert \varphi-{\cal J}_n[\varphi] \Vert _{\mathcal{X}} \leq  (1+\Lambda_n) C_0 \beta_n n^{-\alpha}$ for any $\varphi \in F$, where the parameter $\beta_n$ is defined as in lemma \ref{lemma::hilbert::polynomial}.
\item If, for $n\geq 1$, $d_n \leq C_0 e^{-c_1 n^{\alpha}}$ with $C_0\geq 1$, then $\Vert \varphi-{\cal J}_n[\varphi] \Vert _{\mathcal{X}} \leq 
 (1+\Lambda_n) C_0 \beta_n e^{-c_2 n^{\alpha}}$ for any $\varphi \in F$, where $\beta_n$ and $c_2$ are defined as in lemma \ref{lemma::hilbert::exponential}.
\end{enumerate}
\end{theorem}

%%%%%%%%%%%%%%%%%%%%%%%%%%%%%%%%%%%%%%%%%%

\begin{corollary}
\label{remark3}
If $(\Lambda_n)_{n=1}^\infty$ is a monotonically increasing sequence, then:
\begin{itemize}
\item if $d_n \leq C_0
n^{-\alpha}$ for any $n\geq 1$, then for any $\varphi \in F$,
$$
\Vert \varphi-{\cal
J}_n[\varphi] \Vert _{\mathcal{X}} 
\leq 
\begin{cases}
2C_0 \left( 1+\eta^{-1}\right) (1+\Lambda_1),\quad &\text{if } n=1.
\\
C_0 2^{3\alpha +1} (1+\Lambda_n)^3\eta^{-2}  n^{-\alpha},\quad &\text{if } n\geq 2.
\end{cases}
$$
\item if $d_n \leq C_0 e^{-c_1 n^{\alpha}}$ for $n\geq 1$ and $C_0\geq 1$, then for any $\varphi \in F$, 
$$
\Vert \varphi-{\cal
J}_n[\varphi] \Vert _{\mathcal{X}} 
\leq 
\begin{cases}
2C_0 \left( 1+\eta^{-1}\right) (1+\Lambda_1),\quad &\text{if } n=1,
\\
C_0 \sqrt{2} (1+\Lambda_n)^2\eta^{-1} e^{-c_2 n^{\alpha}},\quad &\text{if } n\geq 2,
\end{cases}
$$
where $c_2 = 2^{-2\alpha-1}$.
\end{itemize}
\end{corollary}

%%%%%%%%%%%%%%%%%%%%%%%%%%%%%%%%%%%%%%%%%%%%%%%%%%%%%%%%%%%
\section{Final remarks} We have analyzed the convergence rates of the interpolation error in GEIM in the case of polynomially or exponentially decaying Kolmogorov $n$-widths of $F$. The impact on this convergence rate of the Lebesgue constant appears as multiplicative factors of order $\ord(\Lambda_n^2)$ or $\ord(\Lambda_n^3)$. Given that, for reasonable enough dictionaries $\Sigma$, it has been observed in practical applications that $(\Lambda_n)_{n=1}^\infty$ is linear in the worst case scenario (see \cite{Barrault}, \cite{GreplMagic}, \cite{MadayMagic}, \cite{mulaStokesGEIM}), our results prove that a decay of order $\ord(n^{-3})$ in 
$d_n(F,\mathcal{X})$ should be enough to ensure the convergence of the interpolation errors of GEIM.  

\appendix
\section{Proof of Theorem \ref{theorem::banach::general}} 
We begin by recalling a preliminary lemma for matrices that is proven in \cite{devore2012greedy}.

\begin{lemma}
\label{conv::lemma::matrix}
Let $G=(g_{i,j})$ be a $K\times K$ lower triangular matrix with rows $\boldsymbol{g_1},\dots,\boldsymbol{g_K}$, $W$ be any $m$ dimensional subspace of $\mathbb{R}^K$, and $P$ be the orthogonal projection of $\mathbb{R}^K$ onto $W$. Then,
\begin{equation}
\prod\limits_{i=1}^{K} g_{i,i}^2
\leq
\left\lbrace
\dfrac{1}{m} \sum\limits_{i=1}^{K} \Vert P\boldsymbol{g_i}  \Vert_{\ell_2}^2
\right\rbrace^m
\left\lbrace
\dfrac{1}{K-m} \sum\limits_{i=1}^{K} \Vert \boldsymbol{g_i}- P\boldsymbol{g_i}  \Vert_{\ell_2}^2
\right\rbrace^{K-m}
\end{equation}
where $\Vert . \Vert_{\ell_2}$ is the euclidean norm of a vector in $\mathbb{R}^K$.
\end{lemma}

For the proof of theorem \ref{theorem::banach::general}, we consider a lower triangular matrix $A = (a_{i,j})_{i,j=1}^\infty$ defined in the following way. For each $j=1,\dots$, we let $\lambda_j \in \mathcal{X}'$ be the bounded linear functional of norm one that satisfies:
\begin{equation}
(i)\ \lambda_j(X_j) =0,
\qquad 
(ii)\ \lambda_j (\varphi_j) = \dist (\varphi_j,X_j),
\end{equation}
where $X_j = \vspan\{\varphi_0,\dots,\varphi_{j-1}\}$, $j=1,2,\dots$, is the interpolating space given by the greedy algorithm of GEIM. The existence of such a functional is a consequence of the Hahn-Banach theorem. We let $A$ be the matrix with entries
\begin{equation*}
a_{i,j} = \lambda_j(\varphi_i).
\end{equation*}
The matrix
$A$ has the following properties:
\begin{itemize}
\item[\textbf{Q1:}] The diagonal elements of $A$ satisfy $\gamma_n \tau_n \leq a_{n,n} \leq \tau_n$.
\item[\textbf{Q2:}] For every $j < i$, $\vert a_{i,j} \vert \leq \dist(\varphi_i,X_j) \leq \tau_j$.
\item[\textbf{Q3:}] For every $j > i$, $a_{i,j}=0$.
\end{itemize}
\begin{proof}
\begin{itemize}
\item[\textbf{Q1:}] We have 
$$a_{j,j} = \lambda_j(\varphi_j) = \dist(\varphi_j,X_j) = \norm{\varphi_j - P_j(\varphi_j)} \leq \underset{\varphi \in F}{\max} \norm{\varphi - P_j(\varphi)} = \tau_j.$$
Lemma \ref{lemma::weakGreedy} directly yields the second part of the inequality:
$
a_{j,j} \geq \gamma_j \tau_j. 
$
\item[\textbf{Q2:}] For any $j < i$ and any $g\in X_j$, 
$$
\vert a_{i,j} \vert
=
\vert \lambda_j(\varphi_i) \vert
=
\vert \lambda_j(\varphi_i -g) \vert
\leq 
\Vert \lambda_j  \Vert_{\mathcal{X}'} \norm{\varphi_j - g},
$$
where we have used the fact that $\lambda_j (g) =0$ because $g\in X_j$. Therefore, since $\Vert \lambda_j  \Vert_{\mathcal{X}'}=1$, we have $\vert a_{i,j} \vert \leq \norm{\varphi_j - g},\quad \forall\ g\in X_j$. Thus, $\vert a_{i,j} \vert \leq \norm{\varphi_i -P_j(\varphi_i)} \leq \tau_j$.
\item[\textbf{Q3:}] Clearly, for $j > i$, $a_{i,j}=\lambda_j(\varphi_i)=0$ because $\varphi_i \in X_j$ in this case.
\end{itemize}
\end{proof}

We can now prove theorem \ref{theorem::banach::general}:
\begin{proof}
For a given $K\geq 2$,  consider the $K\times K$ matrix $G$ formed by the rows and columns of $A$ with indices from $\{N+1,\dots,N+K\}$. Let $Y_m$ be a subspace of $\mathcal{X}$ of dimension $\leq m$ (we recall that $1\leq m < K$). For each $i$,
 there exists an element $h_i \in Y_m$ such that
$$
\norm{\varphi_i - h_i} = \dist(\varphi_i,Y_m) \leq d_{Y_m},
$$ 
where $d_{Y_m} \coloneqq \max_{\varphi \in F} \dist(\varphi,Y_m)$.
% WWW Gabriel -> Olga: j'ai remplace X par Y_m ci dessus.
 Therefore
\begin{equation}
\label{conv::eq::appendixEq1}
\vert \lambda_j(\varphi_i)-\lambda_j(h_i) \vert
=
\vert \lambda_j(\varphi_i-h_i) \vert
\leq
\Vert \lambda_j \Vert_{\mathcal{X}'} \norm{\varphi_i - h_i} \leq d_{Y_m} .
\end{equation}
We now consider the vectors $ \left( \lambda_{N+1}(h),\dots, \lambda_{N+K}(h) \right),\ h\in X_m$. They span a space $W \subset \mathbb{R}^K$ of dimension $\leq m$. We assume that $\dim (W)=m$ (a slight notational adjustment has to be made if $\dim (W)<m$). It follows from \eqref{conv::eq::appendixEq1} that each row $\boldsymbol{g_i}$ of $G$ can be approximated by a vector from $W$ in the $\ell_{\infty}$ norm to accuracy $d_{Y_m}$, and therefore in the $\ell_2$ norm to accuracy $\sqrt{K}d_{Y_m}$. Let $P$ be the orthogonal projection of $\mathbb{R}^K$ onto $W$. Hence, we have
\begin{equation}
\label{conv::eq::appendixEq2}
\Vert \boldsymbol{g_i} - P\boldsymbol{g_i} \Vert_{\ell_2}
\leq
\sqrt{K} d_{Y_m},
\quad i=1,\dots,K.
\end{equation}
Also, from the property Q2, 
$
\Vert P\boldsymbol{g_i} \Vert_{\ell_2} 
\leq
\Vert \boldsymbol{g_i} \Vert_{\ell_2} 
\leq
\left( \sum_{j=1}^{i} \tau_{N+j}^2  \right)^{1/2},
$
and therefore
\begin{equation}
\label{conv::eq::appendixEq3}
\sum\limits_{i=1}^{K}  \Vert P\boldsymbol{g_i} \Vert_{\ell_2}^2
\leq
\sum\limits_{i=1}^{K} \sum\limits_{j=1}^{i} \tau_{N+j}^2
\leq
K \sum\limits_{i=1}^{K} \tau_{N+i}^2.
\end{equation}
Next, we apply lemma \ref{conv::lemma::matrix} for this $G$ and $W$ and use property Q1 and estimates \eqref{conv::eq::appendixEq2} and \eqref{conv::eq::appendixEq3} to derive
\begin{eqnarray*}
\prod\limits_{i=1}^{K} \gamma_{N+i}^2 \tau_{N+i}^2
&\leq &
\left\lbrace  \dfrac{K}{m} \sum\limits_{i=1}^{K} \tau_{N+i}^2 \right\rbrace^{m}
\left\lbrace  \dfrac{K^2}{K-m} d_{Y_m}^2 \right\rbrace^{K-m} \\
&=&
K^{K-m} \left( \dfrac{K}{m} \right)^m  \left( \dfrac{K}{K-m} \right)^{K-m}  \left\lbrace \sum\limits_{i=1}^{K} \tau_{N+i}^2 \right\rbrace^m d_{Y_m}^{2(K-m)} \\
&\leq & 2^K K^{K-m} \left\lbrace \sum\limits_{i=1}^{K} \tau_{N+i}^2 \right\rbrace^m d_{Y_m}^{2(K-m)},
\end{eqnarray*}
where we have used the fact that $x^{-x}(1-x^{x-1})\leq 2$ for $0<x<1$. The proof follows by taking the infimum over all subspaces $Y_m$ of $\mathcal{X}$ of dimension $\leq m$.
\end{proof}

%%%%%%%%%%%%%%%%%%%%%%%%%%%%%%%%%%%%%%%%%%%%%%%%%%%%

\section{Proof of Theorem \ref{theorem::hilbert}}
In this section, $\mathcal{X}$ is a Hilbert space. We denote by $(\varphi_n^*)_{n\geq 0}$ the orthonormal system obtained from $(\varphi_n)_{n\geq 0}$ by Gram-Schmidt orthonormalisation. It follows that the orthogonal projector $P_n$ from $\mathcal{X}$ onto $X_n$ can be written as $P_n\varphi = \sum_{i=0}^{n-1} \ps{\varphi}{\varphi_i^*}\varphi_i^*$, for $n\geq 1$. In particular, $\varphi_n 
= P_{n+1} \varphi_n 
= \sum_{j=0}^{n} a_{n,j} \varphi_j^*$, with $a_{n,j} = \ps{\varphi_n}{\varphi_j^*},\ j\leq n$. There is no loss of generality in assuming that the infinite dimensional Hilbert space $\mathcal{X}$ is $\ell_2 \left( \mathbb{N}\right)$ and that $\varphi_j^*=e_j$, where $e_j$ is the sequence with all entries zero except the $j$-th entry which is $1$. In a similar manner as in the Banach space case, we associate with the greedy procedure of GEIM the lower triangular matrix:
\begin{equation*}
A := (a_{i,j})_{i,j=0}^\infty
,\quad a_{i,j}:=1,\ j>i.
\end{equation*}
The following two properties characterize any lower triangular matrix $A$ generated by such a greedy algorithm. 
\begin{itemize}
\item[\textbf{S1:}] The diagonal elements of $A$ satisfy $\gamma_n \tau_n \leq \vert a_{n,n} \vert \leq \tau_n$. 
\item[\textbf{S2:}] For every $m\geq n$, one has $\sum\limits_{j=n}^{m} a_{m,j}^2 \leq \tau_n^2$.
\end{itemize}

\begin{proof}
\begin{itemize}
\item[\textbf{S1:}] For any $n\geq 1$, since $\varphi_n - P_{n} \varphi_n = a_{n,n} \varphi_n^*$, it follows that 
for any $n\geq 1$, $ \vert a_{n,n} \vert =\Vert \varphi_n - P_n \varphi_n \Vert \leq \tau_n$. The fact that $\vert a_{n,n} \vert \geq \gamma_n \tau_n$ directly follows from lemma \ref{lemma::weakGreedy}.
\item[\textbf{S2:}] For $m\geq n$,
$
\sum\limits_{j=n}^{m} a_{m,j}^2 = \Vert \varphi_m - P_n \varphi_m \Vert^2 \leq \underset{\varphi \in F}{\max} \Vert \varphi - P_n\varphi \Vert^2 = \tau_n^2.
$
\end{itemize}
\end{proof}
We can now prove theorem \ref{theorem::hilbert}:
\begin{proof}
For a given $K\geq 2$, consider the $K \times K$ matrix $G = (g_{i,j})$ formed by the rows and columns of $A$ with indices from $\{N+1,\dots,N+K\}$. Each row $\boldsymbol{g_i}$ is the restriction of $\varphi_{N+i}$ to the coordinates $N+1,\dots,N+K$.  Let $Y_m$ be a subspace of $\mathcal{X}$ of dimension $\leq m$. Then, $\dist(\varphi_{N+i},Y_m) \leq d_{Y_m},\ i=1,\dots,K$. Let $\tilde{W}$ be the linear subspace which is the restriction of $Y_m$ to the coordinates $N+1,\dots,N+K$. In general, $\dim(\tilde{W}) \leq m$. Let $W$ be a $m$ dimensional space, $W \subset \vspan \{ e_{N+1},\dots,e_{N+K}  \}$, such that $\tilde{W} \subset W$ and $P$ and $\tilde{P}$ are the projections in $\mathbb{R}^K$ onto $W$ and $\tilde{W}$, respectively. Clearly,
\begin{equation}
\label{conv::eq::appendixEq5}
\Vert  P \boldsymbol{g_i}  \Vert_{\ell_2}
\leq 
\Vert  \boldsymbol{g_i}  \Vert_{\ell_2}
\leq 
\tau_{N+1}
,\quad i=1,\dots, K,
\end{equation}
where we have used property S2 in the last inequality. Note  that
\begin{equation}
\label{conv::eq::appendixEq6}
\Vert \boldsymbol{g_i} - P \boldsymbol{g_i}  \Vert_{\ell_2}
\leq
\Vert \boldsymbol{g_i} - \tilde{P} \boldsymbol{g_i}  \Vert_{\ell_2}
=
\dist \left( \boldsymbol{g_i} ,\tilde{W} \right)
\leq
\dist \left( \varphi_{N+i} , Y_m \right)
\leq
d_{Y_m}
,\quad i=1,\dots,K.
\end{equation}
It follows from property S1 that 
\begin{equation}
\label{conv::eq::appendixEq7}
\prod\limits_{i=1}^{K} \vert a_{N+i,N+i}  \vert^2
\geq 
\prod\limits_{i=1}^{K} \gamma_{N+i}^2 \tau_{N+i}^2.
\end{equation}
To derive the result, we apply lemma \ref{conv::lemma::matrix} for this $G$ and $W$, use estimates \eqref{conv::eq::appendixEq5}, \eqref{conv::eq::appendixEq6} and \eqref{conv::eq::appendixEq7} and take the infimum over all subspaces of $\mathcal{X}$ of dimension $\leq m$.
\end{proof}

%%%%%%%%%%%%%%%%%%%%%%%%%%%%%%%%%%%%%%%%%%%%%%%%%%%%%%%%%%%%%

\bibliographystyle{siam}
\bibliography{references}

\end{document}